\documentclass[a4paper,11pt]{amsart}
\usepackage[margin=1.2in]{geometry}
\usepackage[utf8]{inputenc}
\usepackage{amsfonts,amssymb,amsmath,amstext,amscd,epsfig}
\usepackage{amsthm,cases,color,graphicx,graphics,psfrag}
\usepackage{amsmath, amssymb, amsthm,pdfsync}
\usepackage[utf8]{inputenc}
\usepackage{graphicx}
\usepackage{enumerate}
\usepackage{mathtools}
\usepackage{xcolor}
\usepackage{cleveref}
\usepackage{esint}
\usepackage{dsfont}
\usepackage{color}
\usepackage{comment}
\usepackage[normalem]{ulem}
\usepackage{environ}
\usepackage{overpic}

\newtheorem{theorem}{Theorem}
\newtheorem{lemma}[theorem]{Lemma}

\newtheorem{prop}[theorem]{Proposition}

\newtheorem{proposition}[theorem]{Proposition}

\newtheorem{remark}[theorem]{Remark}

\newcommand{\vertiii}[1]{{\Big\vert\kern-0.25ex\Big\vert\kern-0.25ex\Big\vert #1 
    \Big\vert\kern-0.25ex\Big\vert\kern-0.25ex\Big\vert}}

\makeatletter

\newcommand*\wthelper[2]{%
        \hbox{\dimen@\accentfontxheight#1%
                \accentfontxheight#11.17\dimen@
                $\m@th#1\widetilde{#2}$%
                \accentfontxheight#1\dimen@
        }%
}
\newcommand*\accentfontxheight[1]{%
        \fontdimen5\ifx#1\displaystyle
                \textfont
        \else\ifx#1\textstyle
                \textfont
        \else\ifx#1\scriptstyle
                \scriptfont
        \else
                \scriptscriptfont
        \fi\fi\fi3
}
\makeatother

\newcommand{\ophim}{ \phi}
\newcommand{\ophic}{\Phi}

\newcommand{\R}{\mathbb{R}}

\newcommand{\cC}{\mathcal{C}}
\newcommand{\cE}{\mathcal{E}}
\newcommand{\cL}{\mathcal{L}}
\newcommand{\cP}{\mathcal{P}}
\newcommand{\cB}{\mathcal{B}}

\renewcommand{\epsilon}{\varepsilon}
\newcommand{\eps}{\varepsilon}
\newcommand{\e}{\eps}

\newcommand{\1}{\mathds{1}}

\newcommand\inner[2]{\left\langle #1 \middle| #2 \right\rangle}
\newcommand{\fdfrac}[2]{\mbox{\footnotesize$\displaystyle\frac{#1}{#2}$}}

\numberwithin{equation}{section}
\numberwithin{theorem}{section}

\crefname{assumption}{assumption}{assumptions}
\crefname{theorem}{theorem}{theorems}
\crefname{lem}{lemma}{lemmas}
\crefname{cor}{corollary}{corollaries}
\crefname{prop}{proposition}{propositions}
\Crefname{theorem}{Theorem}{Theorems}
\crefname{conjecture}{conjecture}{conjectures}
\newcommand{\be}{\begin{equation}}
\newcommand{\ee}{\end{equation}}
\NewEnviron{eqsp}{%
\begin{equation}\begin{split}
  \BODY
\end{split}\end{equation}
}

\begin{document}
\title{The Bramson delay in a Fisher-KPP equation with log-singular non-linearity}

%

\author{Emeric Bouin 
\and  
Christopher Henderson 
}

\begin{abstract}
We consider a class of reaction-diffusion equations of Fisher-KPP type in which the nonlinearity (reaction term) $f$ is merely $C^1$ at $u=0$ due to a logarithmic competition term.  We first derive the asymptotic behavior of (minimal speed) traveling wave solutions that is, we obtain precise estimates on the decay to zero of the traveling wave profile at infinity.  We then use this to characterize the Bramson shift between the traveling wave solutions and solutions of the Cauchy problem with localized initial data. 
We find a phase transition depending on how singular $f$ is near $u=0$ with quite different behavior for more singular $f$.  This is in contrast to the smooth case, that is, when $f \in C^{1,\delta}$, where these behaviors are completely determined by $f'(0)$.  In the singular case, several scales appear and require new techniques to understand.
\end{abstract}

\maketitle

\noindent{\bf Key-Words:} {Reaction-diffusion equations, Logarithmic delay, traveling waves}\\
\noindent{\bf AMS Class. No:} {35K57, 35Q92, 45K05, 35C07}


\section{Introduction}

In this paper, we consider a Fisher-KPP (FKPP) type equation with a rough non-linearity
\begin{equation}\label{e:main}
u_t = u_{xx} + u\left(1 - A\left(\log\left(\frac{\nu}{u}\right)\right)^{1-r}\right),
\end{equation}
for $r>1$, $A>0$, and $\nu = e^{A^{-1/(r-1)}}$.  We note that $\nu$ is a normalization constant that may be removed by scaling; however, it ensures that $1$ is a steady state of~\eqref{e:main}.  Our goal is to understand the effect of the parameter $r$, which quantifies the singularity of the reaction term, on the behavior of solutions.  In particular, we study the shape of the minimal speed traveling wave solutions of~\eqref{e:main} and  the spreading of the solutions of \eqref{e:main} 
when the initial density $u_0$ is localized.

To present our results, we recall the known results for the 
FKPP equation, which is one of the simplest models for the spreading of a population.  It is given by
\begin{equation}\label{e:FKPP}
u_t = u_{xx} + f(u),
\end{equation}
where $f$ satisfies $f(0) = f(1) = 0$, $f' \leq f'(0) = 1$, and $f > 0$ on $(0,1)$.  The typical choice of $f$ is $u(1-u)$; however, other choices arise naturally through connections to branching processes~\cite[eqn~(2a)]{Bramson78}.  We note that the model~\eqref{e:main} arises in connection with the nonlocal FKPP equation (see~\cite{BouinHendersonRyzhik} for the connection to the nonlocal FKPP equation and~\cite{Britton} for the introduction of the nonlocal FKPP equation, which has been studied extensively since).

Starting from to the seminal papers of Fisher \cite{Fisher}, Kolmogorov, Petrovskii and Piskunov \cite{KPP} and Aronson-Weinberger \cite{AronsonWeinberger}, an enormous body of literature has developed studying the existence, qualitative behavior, and stability of traveling wave solutions of~\eqref{e:FKPP}.  In particular, the equation~\eqref{e:FKPP} with $f(u) = u(1-u)$ admits (minimal speed) traveling wave solutions $U$ with speed $2$ such that $U(\xi) \sim \xi e^{-\xi}$ as $\xi\to\infty$~\cite{Hamel}.  In addition, going back to the work of Bramson~\cite{Bramson78,Bramson83} and Uchiyama~\cite{Uchiyama}, it is known that if $u_0$  is compactly supported, 
the front of~$u$ is located at 
\begin{equation}\label{sep2702}
	X(t) =2t - \frac32\log t + s_0,
\end{equation}
where $s_0$ is a shift depending only on $u_0$, in the sense that $u(t,\cdot + X(t)) \to U$.  Interestingly,  the ``delay'' between the traveling wave and $u$, given by $(3/2)\log(t)$ is $o(t)$ but tends to infinity. 
These proofs have been simplified in recent years~\cite{HNRR_Short,Roberts}, with some
refinements to the expansion of $X$ by Nolen, Roquejoffre, and Ryzhik~\cite{NRR1,NRR2} and Graham~\cite{graham2017precise}.

Some of the above results can be extended to more general $f$ as long as it is sufficiently regular near $0$.  In particular, the arguments in~\cite{Hamel,HNRR_Short} require explicitly that there is some $\delta>0$ such that $f$ is $C^{1,\delta}$.  On the other hand, preliminary results in~\cite{BouinHendersonRyzhik} make clear the surprising fact that~\eqref{sep2702} does not hold without this regularity assumption.  In this paper, we are interested in understanding how the regularity of $f$, measured through the parameter $r$, affects the decay of traveling waves and the correction between the traveling wave position $\bar X(t) = 2t$ and the position $X(t)$ of the front for localized initial data.


\subsection*{Setup and main results} 
For the initial condition, we assume\footnote{Notice that, up to shifting the system, we lose no generality from the case where $u_0$ is zero to the right of some $x_0 \in \R$.} that $u_0$ is localized to the left of the origin:
\begin{equation}\label{e:u_0}
    0 \leq u_0 \leq 1,\qquad 
    u_0(x) = 0 \text{ for all } x \geq 0,
    \quad \text{ and } \quad
    \liminf_{x\to-\infty} u_0(x) > 0.
\end{equation}
We expect our results to hold when $u_0$ has ``fast" exponential decay, that is, $u_0(x) e^{(1+\epsilon)x} \to 0$ as~$x\to 0$ for some $\epsilon>0$, rather than compactly
supported on the right, but in this work, we opt for a simpler setting in the interest of clarity.  We recall that the front position asymptotics for 
solutions of (\ref{e:FKPP}) with $u_0$ that has a sufficiently slow exponential 
tail on the right is 
different from (\ref{sep2702}), see~\cite{Bramson78, Bramson83}, and, hence, the ``fast'' exponential decay condition above cannot be weakened further.

Our first result is about the behavior at infinity of critical (minimal speed) traveling wave solutions of \eqref{e:main}; that is, solutions of
\be\label{e.traveling_wave}
- 2 U' =  U'' + U \left(1 -A\left(\log\left(\frac{\nu}{U}\right) \right)^{1-r}\right)
\ee
with the far-field conditions $U(-\infty) = 1$ and $U(+\infty) = 0$.  In the study of~\eqref{e:FKPP}, critical traveling waves attract any initial data $u_0$ satisfying~\eqref{e:u_0} and, at a technical level, understanding their profile at $+\infty$~\cite{Hamel} is crucial in quantifying the position of level sets~\cite{HNRR_Short}. 

\begin{theorem}\label{p:decayTW}
Let $U$ be a traveling wave solution of~\eqref{e:main}, that is solving \eqref{e.traveling_wave}.  Then $U$ is monotonically decreasing, less than $1$, and has the following behavior at infinity:
\begin{enumerate}[(i)]
  \setlength\itemsep{0em}
	\item If $r>3$, then $\displaystyle \lim_{\xi\to\infty} U(\xi) / (\xi e^{-\xi}) = \kappa$ for some $\kappa>0$.
	\item If $r=3$, then $\displaystyle \lim_{\xi\to\infty} U(\xi) / (\xi^\alpha e^{-\xi}) = \kappa$ for some $\kappa>0$ and where $\alpha$ is the unique solution of $\alpha (\alpha-1) = A$ such that $\alpha > 1$.
	\item If $r \in \, (1,3)$, then
	\[
		\lim_{\xi\to\infty} \frac{\log( e^\xi U(\xi))}{\xi^\frac{3-r}{2}}
			= \frac{2\sqrt A}{3-r}.
	\]
\end{enumerate}
\end{theorem}
More importantly, this theorem shows a transition between the standard FKPP regime for $r>3$ and regime $r\leq 3$ for which the nonlinearity plays a crucial role in the decay of the waves. This transition is new in the literature up to our knowledge.

Let us note that if $r \in \, (1,3)$, we obtain a two-term asymptotic expansion of $\log U$:
\[
	\log U(\xi) = - \xi + \frac{2\sqrt{A}}{3-r} \xi^\frac{3-r}{2} + o\left( \xi^\frac{3-r}{2}\right).
\]
Some lower order terms in this expansion are expected, with the number of nontrivial terms increasing as $r$ gets closer to $1$.  The expansion in \Cref{p:decayTW}, however, is sufficient for our purposes.   It is also worth pointing out that some error estimates in \Cref{p:decayTW} could be derived from our analysis. Since it is not our purpose here, we have decided not to pursue this matter further.

The main result of this paper is the following.
\begin{theorem}\label{thm:main_delay}
Suppose that $u$ satisfies~\eqref{e:main} and~\eqref{e:u_0}. 
\begin{enumerate}[(i)]
	\item If $r>3$, then the solution $u$ propagates with a logarithmic delay:
        \begin{align}\label{e:log_delay_below}
        		&\lim_{L\to\infty}\liminf_{t\to\infty} \inf_{x \leq -L} u\Big(t,2t - \frac{3}2 \log t +x\Big)
		= 1  \quad\text{and}\\
\label{e:log_delay_above}
	&\lim_{L\to\infty} \limsup_{t\to\infty} \sup_{x\geq L}\, u\Big(t,2t - \frac{3}{2}\log t + x\Big) = 0.
\end{align}
	\item If $ r= 3$, then the solution $u$ propagates with larger logarithmic delay: defining $\alpha> 1$ to be the solution of $\alpha(\alpha-1) = A$,
	\begin{align}\label{e:weak_log_delay_below}
		&\lim_{L\to\infty}\liminf_{t\to\infty} \inf_{x \leq -L} u\Big(t,2t - \frac{2\alpha+1}{2} \log t + x\Big)
			= 1, \quad\text{and}\\
	\label{e:weak_log_delay_above}
       		&\lim_{x\to\infty} \limsup_{t\to\infty} \, u\Big(t,2t - \frac{2\alpha+1}{2}\log(t) + x\Big) = 0.
    \end{align}
	\item If $r \in (1,3)$, then the delay is algebraic:  there exists $\Theta_r > 0$, such that
\begin{align}\label{e:alg_delay_below}
        &\lim_{t\to\infty}\inf_{x \leq 2t - \Theta_r A^\gamma t^\beta + o(t^\beta)} u(t,x) 
        		= 1  \quad\text{and}\\
	\label{e:alg_delay_above}
        		&\lim_{t\to\infty} \sup_{x \geq 2t - \Theta_r A^\gamma t^\beta - o(t^\beta)} u(t,x) = 0.
    \end{align}
with $\gamma := \frac{2}{1+r}$ and $\beta := \frac{3-r}{1+r}$.    
\end{enumerate}
\end{theorem}

Some comments are now in order. First, the case $r>3$ as already been established in our earlier work with Ryzhik \cite{BouinHendersonRyzhik} and will, thus, not be discussed in the sequel.  We include it in \Cref{thm:main_delay} for the sake of completeness.

Second, each delay coefficient has a heuristic meaning behind it.  As described in the introduction of~\cite{BouinHendersonRyzhik}, the main length scale on which the nonlinearity acts is $x \sim t^\gamma$.  When $r > 3$, then $\gamma < 1/2$ and the diffusive length scale $x\sim t^{1/2}$ dominates, allowing us to ignore the nonlinearity.  In this case, the heuristics are as in the standard case when $f(u) = u(1-u)$: roughly, following~\cite[Section 1]{HNRR_Short}, the $3/2$ coefficient arises because the time decay of the (Dirichlet) heat kernel on the half-line $[0,\infty)$ is $t^{-3/2}$.  When $r=3$, $\gamma = 1/2$ and both scales balance and the nonlinearity is relevant.  In this case, the coefficient $\alpha + 1/2$ comes from the fact that solutions of $h_t = h_{xx} - A x^{-2} h$ have time decay $t^{- \alpha - 1/2}$ (note that, from \Cref{p:decayTW}, we expect $\log(\nu/h)^{1-r} \sim x^{-2}$ when $r=3$).  When $r\in (1,3)$, the scale of the nonlinearity is larger than the scale of the diffusion.  Thus, to understand the dynamics, it is required to understand the large deviations rate function appearing on that scale.  For a more in-depth heuristic discussion of this, we refer to the outset of \Cref{s.lw3}.

Finally, one of the main interests of \Cref{thm:main_delay} is in the fact that all leading order delay constants can be characterized explicitly (when $r\geq 3)$ or implicitly (when $r\in (1,3)$).  In particular, the constant $\Theta_r$, when $r\in (1,3)$, can be easily computed numerically, which is not possible without the reduction we discuss below.  Indeed, in the proof of \Cref{thm:main_delay}.(iii), we show that $\Theta_r$ is given by $\phi(0)$ where $\phi$ solves
\be\label{e.c8251}
\begin{split}
	&\phi' = \frac{\gamma}{2} y - \sqrt{ \frac{\gamma^2 y^2}{4} + A y^{1-r} - \beta \phi}\\
	&\phi\left( \bar y_r \right)
		= \frac{\gamma^2 \bar y_r^2}{4\beta} + \frac{A \bar y_r^{1-r}}{\beta}
			\qquad\text{where } \bar y_r = (1 + r)^\gamma.
\end{split}
\ee
(see \Cref{r:Theta} for this characterization and the beginning of \Cref{s.lw3} for a heuristic explanation of the above characterization).  While, to our knowledge,~\eqref{e.c8251} does not have an explicit solution, it is easy to compute numerically.  Indeed, we see in \Cref{fig:theta}, that $\Theta_r$ appears to be increasing in $r$, with $\Theta_r \searrow 1$ as $r\searrow 1$ and $\Theta_r \sim \alpha/(3-r)$ for some $\alpha>0$ as $r \nearrow 3$.  These asymptotics are an interesting open question that we do not settle here.

\begin{figure}
	\begin{overpic}[width=.49\linewidth]{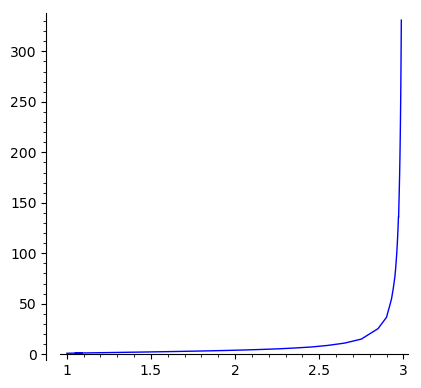}
		\put(3,87){$\Theta_r$}
		\put(50,-1.5){$r$}
	\end{overpic}\
	\begin{overpic}[width=.49\linewidth]{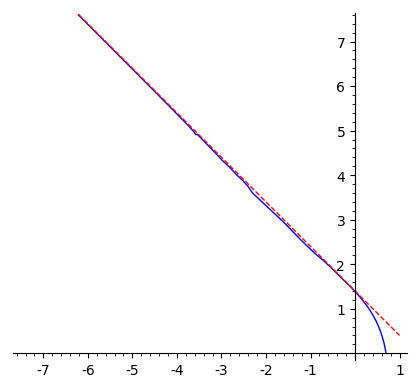}
		\put(87,75){\rotatebox{90}{$\log \Theta_r$}}
		\put(40,-1.5){$\log(3-r)$}
	\end{overpic}\
\caption{The approximate value of $\Theta_r$ as a function of $r \in (1,3)$.  From both plots we see that $\Theta_r$ seems to be increasing in $r$ and tends to $1$ on the left and $\infty$ on the right.  In the second plot, the dashed red curve is a line with slope $-1$.  Hence, we see that $\log\Theta_r \sim -\log(3-r)$ as $r\nearrow 3$.}
\label{fig:theta} 
\end{figure}

\subsection*{Related works} 
The history of the precise asymptotics of the FKPP equation with nonlinearity $u(1-u)$ was given above.  As said briefly earlier, this paper is a companion to \cite{BouinHendersonRyzhik}, in which we had studied the issue of the Bramson correction in the non-local Fisher-KPP equation.  However, the estimates in \cite{BouinHendersonRyzhik} were not explicit and were far from sharp when $r\in(1,3]$.  The connection to the nonlocal FKPP equation is one of the main motivations of the current work in addition to the intrinsically interesting issue of the affect of the particular choice of $f$ on the dynamics of solutions to~\eqref{e:FKPP}.  

We note that explicit algebraic delays are not common in the literature. As far as algebraic delays are concerned, we point to the work of Fang and Zeitouni~\cite{FangZeitouni}, Maillard and Zeitouni~\cite{MaillardZeitouni}, and Nolen, Ryzhik, and Roquejoffre~\cite{NolenRyzhikRoquejoffre} where a Fisher-KPP model with a diffusivity that changes slowly in time was studied, and a delay, roughly, of order $t^{1/3}$ was obtained. However, both the set-up and the mechanism for the large delay are quite different in these papers than in the present work.  Finally, we also mention the recent paper of Ducrot~\cite{Ducrot} in which he constructs a class of non-linearities $f(x,u)$, which tend to $u(1-u)$ as $|x| \to \infty$, such that if the nonlinearity $u(1-u)$ in~(\ref{e:FKPP}) is replaced by~$f(x,u)$, then the front is at $2t - \lambda \log(t)$ for any $\lambda \geq 3/2$.

A bit further from the present setting are the following related threads of research: branching random walk~\cite{AddarioBerryReed, Aidekon, BrunetDerridaBRW, CernyDrewitz}, branching Brownian motion with absorption~\cite{ABBS, BBS11, BBS13, BBHR, HendersonVS}, systems of reaction-diffusion equations~\cite{ChenEtAl, DucrotGiletti}, and inhomogeneous environments and nonlocal interactions~\cite{ABBP, BHR_Log_Delay, BouinHendersonRyzhik, Ducrot, HNRR_Periodic, Penington}.  There is also a higher dimensional analogue to Bramson's result~\cite{Ducrot, Gartner, RRN}.  It is also important to point out the important contributions in the applied literature~\cite{BerestyckiBrunetDerrida, BrunetDerridaFKPP, EbertVanSaarloos}.  Due to the huge amount of research in this area, the previous list is unfortunately quite incomplete.

\subsection*{Organization and notation}

The rest of the paper is organized as follows. The following \Cref{sec:shape_waves} contains the proof of \Cref{p:decayTW}. In \Cref{s.r3}, we present the proof of the case $r=3$: \Cref{thm:main_delay}.(ii). Finally, the case $r\in (1,3)$, that is the proof of \Cref{thm:main_delay}.(iii) is in \Cref{s.lw3}. Note that to keep this introduction relatively short, we have chosen to present all important heuristics at the beginning of each case's dedicated subsection.  Finally, the last section, \Cref{s:technical}, is devoted to a technical lemma, giving rough upper and lower bounds of~\eqref{e:main}, that is used throughout.

We use $C$ to be a positive universal constant, changing line-by-line, that depends only on $r$ and $A$, unless otherwise indicated. Letting $\R_+ = [0,\infty)$, we denote the $L^2(\R_+)$ inner product using the bra-ket notation; that is, for $f,g\in L^2(\R_+)$, we write
\[
	\inner{f}{g}
		= \int_0^\infty f(x) g(x) dx.
\]

We use $\cdot_+$ to denote the positive part operator; that is, for $x\in \R$, we define $x_+ = \max\{0,x\}$.

\subsubsection*{Acknowledgements}

CH was partially supported by NSF grant DMS-2003110.

\section{The shape of traveling waves - Proof of \Cref{p:decayTW}}
\label{sec:shape_waves}

\begin{proof}
The fact that $U$ is less than $1$ follows from the maximum principle, and the fact that $U$ is decreasing follows directly from standard arguments, see \cite{AronsonWeinberger}. 
We remove an exponential factor, letting $Q = \nu e^\xi U$.  Then $Q$ satisfies
\be\label{e.c7132}
	Q'' 
		= A \left( \xi - \log(Q) \right)^{1-r} Q.
\ee
Note that $\xi - \log(Q) = \log(\nu/U) \geq 0$ since $\nu > 1 > U$ and, hence, $Q''> 0$.  It follows that
\be\label{e.c7131}
	Q(\xi) \geq Q'(\xi_0)(\xi-\xi_0)+Q(\xi_0)
		\qquad\text{ for any } \xi, \xi_0 \in \R.
\ee
Since $U = Qe^{-\xi}$ is bounded as $\xi \to-\infty$, we have that $Q(\xi) \to 0$ as $\xi\to-\infty$.  It follows that the derivative $Q'$ is necessarily nonnegative everywhere; indeed, if $Q'(\xi_0) < 0$ for some $\xi_0$, then the inequality~\eqref{e.c7131} yields a contradiction as $\xi \to -\infty$. As a consequence, $Q$ and $Q'$ are strictly increasing.

Moreover, since $Qe^{-\xi}$ goes to $0$ at $+\infty$, we have that $\xi - \log(Q) \to \infty$ as $\xi \to +\infty$.  Using~\eqref{e.traveling_wave}, it follows that $Q'' / Q \to 0$ as $\xi \to +\infty$.  Using this with the fact that $Q'/Q \geq 0$ and satisfies
\[
	\left(\frac{Q'}{Q}\right)'
		= \frac{Q''}{Q} - \left( \frac{Q'}{Q}\right)^2,
\]
we deduce that
\be\label{e.c7148}
	\lim_{\xi\to\infty} \log(Q)'
		= \lim_{\xi\to\infty} \frac{Q'}{Q}
		= 0.
\ee.


{\bf Case (i): $r>3.$} Using the observations above, fix $\eps \in (0,1)$ and $\xi_0$ sufficiently close to $-\infty$ such that $\vert \log(Q) \vert < \eps \xi$ for all $\xi$ larger than $\xi_0$. Integrating~\eqref{e.c7132}, we find, for any $\xi>\xi_0$,
\be\label{e.c7133}
\begin{split}
Q'(\xi) &\leq Q'(\xi_0) + A (1-\eps)^{1-r} \int_{\xi_0}^\xi x^{1-r} Q(x) \, dx \\
	&= Q'(\xi_0) + A (1-\eps)^{1-r} \left( \frac{\xi^{2-r}}{2-r} Q(x)-  \frac{\xi_0^{2-r}}{2-r} Q(\xi_0) - \int_{\xi_0}^\xi \frac{x^{2-r}}{2-r} Q'(x) \, dx \right) \\
	&\leq Q'(\xi_0) +  A (1-\eps)^{1-r}  \frac{\xi_0^{2-r}}{r-2} Q(\xi_0) +  A (1-\eps)^{1-r} \int_{\xi_0}^\xi \frac{x^{2-r}}{r-2} Q'(x) \, dx,\\
	&= Q'(\xi_0) +  A (1-\eps)^{1-r} \int_{\xi_0}^\xi \frac{x^{2-r}}{r-2} Q'(x) \, dx,
\end{split}
\ee
where we have used the non-negativity of $Q$ and that $r > 3> 2$. Using the Gr\"onwall lemma in its integral form, we find, for all $\xi \geq \xi_0$,
\begin{equation*}
	Q'(\xi) \leq Q'(\xi_0) \exp\left\{  \int_{\xi_0}^\xi A (1-\eps)^{1-r}  \frac{x^{2-r}}{r-2}\, dx\right\}
		\leq Q'(\xi_0) \exp\left\{ \frac{A(1-\eps)^{1-r}}{(r-3)(r-2)} \xi_0^{3-r}\right\}.
\end{equation*}
The last inequality follows from explicitly computing the integral in the exponential and using that $r> 3$.  As a consequence, $Q'$ is bounded above as $\xi\to\infty$.

To summarize, $Q'$ is positive, increasing, and bounded above.  It follows that $Q'(\xi) \to \kappa$ for some $\kappa > 0$ as $\xi \to \infty$.  This yields the result.


%
%
%

{\bf Case (ii): $r = 3$.}  
%
We define $\alpha$ to be the solution of $\alpha(\alpha-1) = A$ such that $\alpha > 1$. Writing, for positive $\xi$, $Q=\xi^\alpha M$ and using~\eqref{e.c7132}, we find
\be\label{e.c7143}
	M'' + \frac{2\alpha}{\xi} M'
		= \frac{A}{\xi^2} \left[ \left(1- \frac{\log(Q)}{\xi} \right)^{-2} -1 \right] M.
\ee
We now prove that $M$ is convergent. Fix $\xi_0>0$.  Notice that the left hand side above is equal to $\xi^{-2\alpha} (M' \xi^{2\alpha})'$.  Hence, multiplying both sides by $\xi^{2\alpha}$ and integrating gives, for all $\xi\geq \xi_0$,
\be\label{e.c7144}
	M'(\xi)
		= \xi_0^{2\alpha} M'(\xi_0) \xi^{-2\alpha} + \xi^{-2\alpha}\int_{\xi_0}^\xi A \left[ \left(1-  \xi'^{-1}\log(Q)(\xi') \right)^{-2} -1 \right] \xi'^{2(\alpha-1)} M \, d\xi'.
\ee
Recalling that $\alpha > 1$ and integrating once again yields
\be\label{e.c7146}
\begin{split}
M(\xi)
	&= M(\xi_0) + \xi_0^{2\alpha} M'(\xi_0) \left( \frac{\xi_0^{1-2\alpha}}{2\alpha-1}- \frac{\xi^{1-2\alpha}}{2\alpha-1}\right)\\
		&\qquad + \int_{\xi_0}^\xi x^{-2\alpha}\int_{\xi_0}^x A \left[ \left(1-  \xi'^{-1}\log(Q) \right)^{-2} -1 \right] \xi'^{2(\alpha-1)} M \, d\xi' \, dx,\\
	&\leq K_0 - \frac{\xi^{1-2\alpha}}{2\alpha-1}\int_{\xi_0}^\xi A \left[ \left(1-  \xi'^{-1}\log(Q) \right)^{-2} -1 \right] \xi'^{2(\alpha-1)} M \, d\xi'  \\
		&\qquad + \int_{\xi_0}^\xi \frac{\xi'^{1-2\alpha}}{2\alpha-1} A \left[ \left(1-  \xi'^{-1}\log(Q) \right)^{-2} -1 \right] \xi^{2(\alpha-1)} M \, d\xi' \\
	&\leq K_0
		+ \frac{A}{2\alpha-1} \int_{\xi_0}^\xi \xi'^{-1}\left[ \left(1-  \xi'^{-1}\log(Q) \right)^{-2} -1 \right] M \, d\xi',
\end{split}
\ee
where
\[
	K_0 = M(\xi_0) + \frac{\xi_0 M'(\xi_0)}{2\alpha-1}
		< \infty
\]
and we have used that $M$ and $\log(Q)$ are positive and thus
\[
	M\left(\left(1-  \xi'^{-1}\log(Q) \right)^{-2} -1\right) \geq 0.
\]
By the Gr\"onwall lemma, one gets that
\be\label{e.c7142}
	M(\xi)
		\leq K_0 \exp\left\{ \frac{A}{2\alpha-1} \int_{\xi_0}^\xi  \left[ \left(1-  \xi'^{-1}\log(Q) \right)^{-2} -1 \right] \xi'^{-1} \, d\xi'\right\}.
\ee

We use~\eqref{e.c7142} twice, where we improve our bound each time we use it.  First, we obtain a preliminary bound to show that $\xi'^{-1} \log Q$ tends to zero.  From~\eqref{e.c7148}, up to increasing $\xi_0$, we have that $|\log Q(\xi)| \leq \xi / 2$ for all $\xi \geq \xi_0$.  Hence,~\eqref{e.c7142} becomes
\[
	M(\xi)
		\leq K_0 \exp\left\{ \frac{3A}{2\alpha-1} \int_{\xi_0}^\xi \xi'^{-1} \, d\xi'\right\}
		\leq K_0 \left(\frac{\xi}{\xi_0}\right)^\frac{3A}{2\alpha-1}.
\]
In particular, this yields $C>0$ such that, if $\xi\geq \xi_0$,
\[
	\log Q(\xi)
		= \log( \xi^\alpha M(\xi))
		\leq C \log(\xi).
\]

We plug this bound back into~\eqref{e.c7142} in order to obtain a uniform bound on $M$.  Indeed, using a Taylor expansion, we find, up to increasing $\xi_0$,
\be\label{e.c7145}
	\left(1-  \xi^{-1}\log Q(\xi) \right)^{-2} -1 \leq  \frac{4C \log(\xi)}{\xi}.
\ee
Plugging this into~\eqref{e.c7142}, we obtain the improved bound:
\[
	M(\xi)
		\leq K_0 \exp\left\{ \frac{4AC}{2\alpha-1} \int_{\xi_0}^\xi  \frac{\log(\xi)}{\xi^2} \, d\xi'\right\}.
\]
Thus, $M$ is uniformly bounded for $\xi \geq \xi_0$.  Let $M_\infty = \|M\|_{L^\infty(\xi_0,\infty)}$.

From~\eqref{e.c7144} and~\eqref{e.c7145}, we see
\be\label{e.c7147}
	\left|\xi^{2\alpha} M'(\xi)\right|
		\leq \left|M'(\xi_0) \xi_0^{2\alpha}\right|
			+ 4 AC M_\infty \int_{\xi_0}^\xi \frac{\log(\xi)}{\xi^2}\, d\xi.
\ee
It follows from the above and the fact that $\alpha > 1$, that $M'$ is integrable as $\xi\to\infty$.  We conclude that $M(\xi) \to \kappa$ as $\xi\to\infty$ for some $\kappa \geq 0$.

We are finished if we show that $\kappa > 0$.  To this end, we argue by contradiction, assuming that $\kappa = 0$.  Taking $\xi \to\infty$ in the first line of~\eqref{e.c7146}, we find, for all $\xi_0>0$,
\[
	M(\xi_0)
		\leq - \frac{\xi_0 M'(\xi_0)}{2\alpha -1}.
\]
Including~\eqref{e.c7147} in the above yields, for some $C>0$ and all $\xi_0$ sufficiently large,
\[
	M(\xi_0)
		\leq C\xi_0^{1 - 2\alpha}
\]
Returning to $Q$, we deduce that $Q(\xi_0) \to 0$ as $\xi_0 \to\infty$, which contradicts the fact that $Q$ is increasing and everywhere positive.  Thus, we conclude that $\kappa>0$, which finishes the proof.


%
%

{\bf Case (iii): $r \in (1,3)$.}    
In this case, we expect sub-exponential growth at infinity. As a consequence, we define $W= \log(Q)$.  Notice that $W>0$ for all $\xi \gg 1$ and $W'>0$ for all $\xi$ since $Q$, $Q'$, and $Q''$ are positive. Then $W$ satisfies
\begin{equation*}
W'' + \vert W' \vert^2 = A(\xi- W)^{1-r},
\end{equation*}
where we know that $\lim_{\xi\to\infty} W(\xi)/\xi = 0$ by~\eqref{e.c7148}.

Let $V = W'$.  This satisfies
\be\label{e.c7149}
V' + \vert V \vert^2 = f^2,
\ee
where $f(\xi) = A^\frac12(\xi-W)^{\frac{1-r}{2}}$. From the above, it is heuristically clear that, as $\xi\to\infty$, $V \sim f$.  We aim to show this.

First, we claim that $V\geq f$ for all $\xi$ sufficiently large.  To prove this, first notice that $f' < 0$ for $\xi$ sufficiently large by~\eqref{e.c7148}.   Then, whenever $V < f$, $V$ is increasing.  Since $V >0$ and $f\to 0$ as $\xi\to\infty$, it follows that there exists $\xi_0 > 0$ such that $V(\xi_0) > f(\xi_0)$ and $f' < 0$ on $[\xi_0,\infty)$.  Let $\xi_1 > \xi_0$ be the first time that $V(\xi_1) = f(\xi_1)$.  If no such $\xi_1$ exists, we are finished.  Otherwise, we have that $V'(\xi_1) \leq f'(\xi_1)$ since $\xi_1$ is the location of a minimum of $V - f$ on $[\xi_0,\xi_1]$.  However, by the choice of $\xi_1$ and~\eqref{e.c7149}, $V'(\xi_1) = 0$ and $f'(\xi_1) < 0$.  This is a contradiction.  We conclude that $V\leq f$ for all $\xi$ sufficiently large.

Let $Y = \frac{V}{f} -1$.  Then $Y$ satisfies
\begin{equation*}
	Y' 
		=  - Y \left( (2+Y) f  - \frac{f'}{f} \right)  -  \frac{f'}{f}.
\end{equation*}
By the paragraph above, we find $\xi_0$ such that $Y$ is nonnegative on $[\xi_0,\infty)$.  Then, on this domain,
\begin{equation*}
	Y'
		\leq - Y \left( 2f  - \frac{f'}{f} \right)  -  \frac{f'}{f},
\end{equation*}
Solving this differential inequality yields, for $\xi \geq \xi_0$,
\[
	\begin{split}
		Y(\xi)
			&\leq \exp\left\{ - \int_{\xi_0}^\xi \left(2f - \frac{f'}{f}\right)\, dx\right\}\left( Y(\xi_0)
				- \int_{\xi_0}^\xi \frac{f'}{f} \exp\left\{\int_{\xi_0}^{x} \left(2f - \frac{f'}{f}\right)\, dx'\right\} dx\right)\\
			&= \frac{f(\xi)}{f(\xi_0)} \exp\left\{ - \int_{\xi_0}^\xi 2f\, dx\right\}\left( Y(\xi_0)
				- f(\xi_0) \int_{\xi_0}^\xi \frac{f'(x)}{f(x)^2} \exp\left\{\int_{\xi_0}^{x} 2f\, dx'\right\} dx\right)
	\end{split}
\]

By increasing $\xi_0$, we have $0 \leq W \leq \xi/2$ on $[\xi_0,\infty)$.  Hence, letting $\beta_0 = 2^{-\frac{1-r}{2}}$,
\be\label{e.f_approx}
	\sqrt A \beta_0 \xi^\frac{1-r}{2}
		\leq f(\xi)
		\leq \sqrt A \xi^\frac{1-r}{2}
		\qquad \text{ for all } \xi \geq \xi_0.
\ee
Plugging this into the above (recalling that $f' \leq 0$), we find
\[\begin{split}
	Y(\xi)
			&\leq \frac{f(\xi)Y(\xi_0)}{f(\xi_0)} \exp\left\{ - 2 \sqrt A \beta_0 \int_{\xi_0}^\xi x^\frac{1-r}{2} \, dx\right\}\\
			&\qquad
				- f(\xi) \int_{\xi_0}^\xi \left(\frac{1}{f(x)}\right)' \exp\left\{-2 \sqrt A \beta_0\int_x^\xi x'^{\frac{1-r}{2}}\, dx'\right\} dx\\
			&= \frac{f(\xi)Y(\xi_0)}{f(\xi_0)} \exp\left\{ - \frac{4 \sqrt A \beta_0}{3-r} \left( \xi^\frac{3-r}{2} - \xi_0^\frac{3-r}{2}\right)\right\}\\
				&\qquad + f(\xi) \int_{\xi_0}^\xi \frac{r-1}{2\sqrt A} (x - W)^{- \frac{3-r}{2}} (1 - W') \exp\left\{- \frac{4 \sqrt A \beta_0}{3-r} \left( \xi^\frac{3-r}{2} - x^\frac{3-r}{2}\right)\right\} dx.
\end{split}\]
Further increasing $\xi_0$ if necessary, we have $(x - W)^{-\frac{3-r}{2}} (1 - W') \leq \beta_0^{-1} x^{-\frac{3-r}{2}}$.  Hence, we find
\[\begin{split}
	Y(\xi)
		&\leq \frac{f(\xi)Y(\xi_0)}{f(\xi_0)} \exp\left\{ - \frac{4 \sqrt A \beta_0}{3-r} \left( \xi^\frac{3-r}{2} - \xi_0^\frac{3-r}{2}\right)\right\}\\
			&\qquad  + f(\xi) \int_{\xi_0}^\xi \frac{r-1}{2\sqrt A \beta_0} x^{- \frac{3-r}{2}} \exp\left\{- \frac{4 \sqrt A \beta_0}{3-r} \left( \xi^\frac{3-r}{2} - x^\frac{3-r}{2}\right)\right\} dx.
\end{split}\]
The first term clearly tends to zero.  Hence, we focus on the second term.  Changing variables and using~\eqref{e.f_approx} yields
\[\begin{split}
	f(\xi) \int_{\xi_0}^\xi &\frac{r-1}{2\sqrt A \beta_0} x^{- \frac{3-r}{2}} \exp\left\{- \frac{4 \sqrt A \beta_0}{3-r} \left( \xi^\frac{3-r}{2} - x^\frac{3-r}{2}\right)\right\} dx\\
		&\leq \frac{r-1}{\beta_0} \int_{\frac{\xi_0}{\xi}}^1 u^{- \frac{3-r}{2}} \exp\left\{- \frac{4 \sqrt A \beta_0}{3-r} \xi^\frac{3-r}{2} \left( 1 - u^\frac{3-r}{2}\right)\right\} du.
\end{split}\]
Applying the Lebesgue Dominated Convergence Theorem, we conclude that the left hand side tends to zero.  Putting together all estimates above, we find that
\be\label{e.c7151}
	\lim_{\xi\to\infty} Y(\xi) = 0.
\ee

We now show how to conclude.  Fix any $\epsilon>0$ and find $\xi_0$ such that if $\xi \geq \xi_0$, then
\[
	\left|V(\xi) - \sqrt A \xi^\frac{1-r}{2}\right|
		\leq \epsilon \xi^\frac{1-r}{2}.
\]
This follows from~\eqref{e.c7151}, the definition of $f$, and the fact that $W(\xi) / \xi \to 0$ as $\xi\to\infty$.  Then, for any $\xi \geq \xi_0$,
\[\begin{split}
	\left| W(\xi) - \frac{2 \sqrt A}{3-r} \xi^\frac{3-r}{2}\right|
		& = \left|W(\xi_0) + \int_{\xi_0}^\xi (V(x) - \sqrt A x^\frac{1-r}{2}) dx \right|\\
		&\leq W(\xi_0) + \eps \int_{\xi_0}^\xi x^\frac{1-r}{2} dx
		\leq W(\xi_0) + \frac{2 \eps}{3-r} \xi^\frac{3-r}{2}.
\end{split}\]
The result follows from the above.  This concludes the proof.
\end{proof}

\section{The case $r=3$: \Cref{thm:main_delay}.(ii)}\label{s.r3}

Throughout this section $r=3$ always, even when not explicitly stated.  The proof proceeds in two main parts: the lower bound and then the upper bound on the delay term.  Before embarking on the proof, we state a weak estimate of $u$, obtained by approximating with solutions of the heat equation, that is required to control behavior for small times.  This is proved in~\Cref{s:technical}.
\begin{lemma}\label{l.heat_bounds}
	Suppose that $u: [0,\infty)\times\R \to [0,1]$ satisfies
	\be\label{e:hb}
		u_t = u_{xx} + f(u),
	\ee
	where $f$ enjoys the bounds $0 \leq f(u) \leq u$ for all $u\in [0,1]$.  If $u_0$ satisfies~\eqref{e:u_0}, then there is a universal constant $C>0$ such that, for all $t\geq 1$ and $x\in \R$
	\[
		\frac{\sqrt t}{C (x_++\sqrt t)} e^{- \frac{x_+^2}{4t} - \frac{C x_+}{t}}
			\leq u(t, x)
			\leq \frac{C\sqrt t}{x_+ +\sqrt t} e^{t - \frac{x_+^2}{4t}}.
	\]
\end{lemma}

Throughout this section, we refer to the shift constant $\alpha + 1/2$ often.  It is, thus, useful to have a shorthand for this constant.  As such, we let
\[
	S_A = \alpha + \frac{1}{2}.
\]

\subsection{Lower bound for the Bramson correction: proof of~\eqref{e:weak_log_delay_below}}
\label{sec:bramsonlower}

%

Miraculously, we can construct two subsolutions (one in the ``beyond the front'' regime and the other in the ``at and to the left of the front'' regime) that, when glued together, give us the sharp asymptotics.

Along the way, we obtain a slightly weaker lower bound on $u$ that holds for all $\geq 1$. This is useful in our proof of the upper bound in the sequel.  We state this as a proposition here; its proof is obtained along the way while establishing~\eqref{e:weak_log_delay_below}, below.

\begin{proposition}\label{p:lower_bound_r3}
	Given $u$ satisfying~\eqref{e:main}, there exists $C$, depending only on the initial data such that, for all $t\geq 1$,
	\begin{equation*}
		\frac{1}{C} \left( \1_{(-\infty,0]}(x) + (1 + x)^\alpha \1_{(0,\infty)}(x) e^{-x - \frac{x^2}{4t}}\right)
			\leq u(t,x + (2t - S_A\log(t))).
	\end{equation*}
	Here $\alpha>1$ is the parameter given in~\eqref{p:decayTW}.
\end{proposition}

We now prove the lower bound in \Cref{thm:main_delay}.(ii).

\begin{proof}[Proof of~\eqref{e:weak_log_delay_below}]

	We proceed in multiple steps outlined here.  First, we change to the moving frame corresponding to the above bound.  Then, we construct a subsolution of~\eqref{e:main} in the moving from $x\mapsto x+ 2t$.  After, we ``fit'' this under the function $u$ at a future time in order to get a bound for any $x \geq 2t$.  The final step is to bootstrap this up to a bound on $[2t-(\alpha+1/2) \log(t+1), 2t]$ using a shift of the traveling wave as a subsolution to ``pull back'' the previous bound.
	
	For the first step, let $u(t,x) = e^{- (x - 2t)} \tilde u(t,x- 2t)$.
	We see that
	\be\label{e.c7241}
		\tilde u_t = \tilde u_{xx} - A(x + \log(1/\tilde u))^{-2}.
	\ee
	
	The next step is to obtain a lower bound of $u$ on $[2t,\infty)$ via the construction of a subsolution of $\tilde u$.  To this end, let, for all $x\geq 0$ and $t \geq 1$,
	\[
		\underline u(t,x) = \e \frac{x^\alpha}{(t-1/2)^{S_A}} e^{-\frac{x^2}{4 (t-1/2)}},
	\] 
	where $\e>0$ is chosen such that $\underline u < 1$ for all $x \geq0$ and $t\geq 1$.  Recall that $S_A = \alpha + 1/2$.  Then
	\begin{equation}\label{e:ss_c1}
	\begin{split}
		\underline u_t - \underline u_{xx} + A(x  + \log(1/\underline u))^{-2}
			&=\frac{ \underline u}{x^2} \left[ \frac{Ax^2}{(x+\log(1/\underline u))^2} - \alpha(\alpha-1)\right]\\
			&= \frac{ A\underline u}{x^2} \left[\frac{x^2}{(x+\log(1/\underline u))^2} - 1\right],
	\end{split}
	\end{equation}
	where the second equality follows from the fact that, by definition, $\alpha(\alpha-1) = A$.  Since $\underline u < 1$, we have $(x + \log(1/\underline u))^2 > x^2$.  Thus
	\be\label{e.c8311}
		\underline u_t - \underline u_{xx} + A(x  + \log(1/\underline u))^{-2}
			\leq 0.
	\ee

	Using~\Cref{l.heat_bounds} we have that, for all $x\geq 0$,
	\[
		\tilde u(1,x)
			\geq \frac{e^x}{C} e^{- \frac{(x+ C + 2)^2}{4}}
			\geq \frac{1}{C} e^{- \frac{x^2}{4} - \frac{C^2}{4} - 1 - \frac{C x}{2} - C}.
	\]
	Then, up to decreasing $\e$ further, we find
	\[
		\underline u(1, x)
			= \e 2^{S_A} x^\alpha e^{- \frac{x^2}{2}}
			\leq \tilde u(1, x).
	\]
	Hence, recalling~\eqref{e.c8311}, the comparison principle implies that $\underline u(t,x) \leq \tilde u(t,x)$ for all $t\geq 1$ and $x\geq 0$.  This establishes that, for some $C>0$,
	\begin{equation}\label{e:ss_c2}
		u(t,x + 2t)
			\geq \frac{1}{C} \frac{x^\alpha}{t^{S_A}} e^{-x - \frac{x^2}{4(t-1/2)}}
				\qquad \text{ for all } t\geq 1, x \geq 0.
	\end{equation}
	
	We have a lower bound on $u$ on $[2t,\infty)$; however, we do not have control on $(-\infty,2t]$.  Let $U$ be the (speed 2 traveling wave) solution of
	\be\label{e.c7242}
		- 2U_\e' = U_\e'' +  U_\e\left(1 - A \log\left(\frac{\nu-\e}{U_\e}\right)^{-2}\right),
	\ee
	for any fixed $\e \in (0,\nu)$, satisfying $U_\e(+\infty) = 0$.  In this case, $U_\e(-\infty)$ satisfies $A \log((\nu-\e)/U_\e(-\infty))^{-2} = 1$ and is, thus, less than $1$.  In addition, by \Cref{p:decayTW}, $U_\e$ is decreasing.
	
	Fix $L\geq 0$ to be chosen. For all $t\geq 1$ and $x \geq 0$, let
	\[
		v(t,x)
			= U_\e\left(x - 2t + S_A \log(t)+L\right).
	\]
	A direct computation using~\eqref{e.c7242} yields
	\[
		v_t - v_{xx} - v\left(1 - A \log \left(\frac{\nu}{v}\right)^{1-r}\right)
			= \frac{S_A}{t} v_x
				+ A v \left(\log \left(\frac{\nu}{v}\right)^{-2}
					- \log \left(\frac{\nu-\e}{v}\right)^{-2} \right)
	\]
	Since $U_\e$ is decreasing, $v_x < 0$.  In addition, $\nu - \e >\nu$, so so parenthetical term above is negative as well.  It follows that $v$ is a subsolution of~\eqref{e:main}.
	
	We now show that, at a positive time $t_0$ to be determined, $v(t_0,\cdot)\leq u(t_0,\cdot)$ on $(-\infty,2t+\sqrt t]$.  First, we check the left hand boundary $x = -\infty$.  By standard theory~\cite{Xin_Book},
	\[
		\lim_{t\to\infty} \inf_{x\leq 0} u(t,x) = 1.
	\]
	Hence, there is $t_0 \geq 1$ sufficiently large that $u(t_0,x) \geq U(-\infty)$ for all $x \leq 0$.
	
	Next we check the right hand boundary $x = 2t$.  On the one hand, we have, by~\eqref{e:ss_c2} that
	\[
		u(t,2t + \sqrt t)
			\geq \frac{1}{C t^{\frac{\alpha + 1}{2}}} e^{-\sqrt t - \frac{1}{4(1 - (2\sqrt t)^{-1})}}.
	\]
	On the other hand, using \Cref{p:decayTW}, there is $\kappa>0$ such that, for all $t$ sufficiently large,
	\[\begin{split}
		v(t,2t + \sqrt t)
			&= U_\e(\sqrt t - S_A \log(t))
			\leq 2\kappa (\sqrt t - S_A \log(t)+L)^\alpha e^{- \left(\sqrt t + (\alpha + 1/2)\log(t) + L\right)}\\
			&= 2\kappa (\sqrt t - S_A \log(t)+L)^\alpha t^{- \alpha - \frac{1}{2}} e^{- \sqrt t - L}
			\leq 2\kappa \left(1+\frac{L}{\sqrt t}\right)^\alpha t^{- \alpha - \frac{1}{2}} e^{- \sqrt t - L}.
	\end{split}\]
	
	It follows from the work in the previous paragraph that, up to increasing $L$,
	\[
		u(t_0, 2t_0 + \sqrt{t_0})
			\geq v(t_0, 2 t_0 + \sqrt{t_0}).
	\]
	All arguments above show that $v\leq u$ on the parabolic boundary of $\{(t,x) : t \geq t_0, x \in (-\infty, 2t + \sqrt t)\}$ and $v$ is a subsolution of $u$ on this set.  We may, thus, apply the comparison principle to conclude that $v \leq u$ on $(-\infty, 2t + \sqrt t)$.  In particular, we find that
	\be\label{e.c7243}
		u(t, 2t - S_A \log(t) + x)
			\geq U_\e(x).
	\ee
	
	We note that, taking $\e$ sufficiently close to $\nu$ depending only on $\liminf_{x\to-\infty} u_0(x)$, we can take $t_0=1$ and all work above can easily be modified in this case to conclude~\eqref{e.c7243}.  In this case, the combination of~\eqref{e:ss_c2} and~\eqref{e.c7243} yield \Cref{p:lower_bound_r3}.

	Returning to the proof at hand, we use~\eqref{e.c7243} to conclude that, choosing $R_\e$ such that $U_\e(R_\e) = U_\e(-\infty) - \e$,
	\[
		\inf_{t \geq t_0, x\leq -R_\e} u(t,x + 2t - (\alpha+1/2)\log(t))
			\geq \inf_{x\leq -R_\e} U_\e(x)
			= U_\e(-R_\e)
			= U_\e(-\infty) - \e.
	\]
	Since $U_\e(-\infty) \to 1$ as $\e \to 0$, the result follows.  The proof is finished.
\end{proof}

\subsection{The upper bound for the Bramson correction when $r=3$}

\subsubsection{A supersolution and the proof of~\eqref{e:weak_log_delay_above}}

Our starting point is the lower bound \Cref{p:lower_bound_r3}.  We use this  in order to construct a sharp supersolution.  To do so we proceed in several steps, the first of which is to consider a new problem, of which $u$ is a sub-solution.  For $t_0$ to be determined, let
\[
	v(t,x)
		= \nu^{-1} e^x u(t, x + 2t - S_A\log(t/t_0)).
\]
Recall that $S_A = \alpha + 1/2$, where $\alpha>1$ solves $\alpha (\alpha-1) = A$.  Indeed, using \Cref{p:lower_bound_r3}, we have that, for all $(t,x)$ such that $t\geq 1$ and $x \geq 0$,
\begin{equation}\label{e:ss_c10}
\begin{split}
	0
		&= v_t + \frac{S_A}{t} (v_x - v) - v_{xx} + Av \log(1/ e^{-x} v)^{-2}\\
		&\leq v_t + \frac{S_A}{t} (v_x - v) - v_{xx} + A\frac{v}{ (x + \frac{x^2}{4t} - \alpha \log(1+x) + \log C_0 )^2}.
\end{split}
\end{equation}
Without loss of generality, we may assume that $C_0$ is sufficiently large that $(1+x)^\alpha e^{-x - x^2/4t}/ C_0 < 1/2$ and, hence, that
\[
	x + x^2/4t - \alpha \log(1+x) + \log C_0  > \log(2),
\]
so that last term in~\eqref{e:ss_c10} is well-defined.

\begin{proposition}\label{p:supersoln_r3}
	There exists $t_0, C_0 \geq 1$, $C>0$, and a smooth function $\overline v: [t_0,\infty)\times[0,\infty) \to \R_+$ that satisfy the following:
	\begin{enumerate}[(i)]
	\item $\overline v$ solves
	\[
		\overline v_t + \frac{S_A}{t} (\overline v_x - \overline v) - \overline v_{xx} + A\frac{\overline v}{ (x + \frac{x^2}{4t} - \alpha \log(1+x) + \log C_0 )^2}
			= 0,
	\]
	\item $\overline v(t_0,x) \geq v(t_0,x)$ for all $x \geq 1$,
	\item $\overline v(t,1) \geq v(t,1)$ for all $t \geq t_0$, and
	\item $\overline v(t,x) \leq C (1+x)^\alpha$ 
		for all $t\geq t_0$ and $x\geq 0$.
	\end{enumerate}
\end{proposition}

Before proving the proposition, we show how to conclude the upper bound on $u$ from \Cref{p:supersoln_r3}.
\begin{proof}[Proof of~\eqref{e:weak_log_delay_above}]
	%
	We first show that $\overline v \geq v$ via the comparison principle.  In order to do this, we notice that $\overline v$ is a super-solution of the equation that $v$ is a sub-solution of by \Cref{p:supersoln_r3}.(i) and~\eqref{e:ss_c10}.  In addition, $\overline v \geq v$ on the parabolic boundary of $[t_0,\infty)\times[1,\infty)$ by \Cref{p:supersoln_r3}.(ii) and (iii).  Hence, the comparison principle and the definition of $v$ imply that
	\[
		\overline v(t,x) \geq v(t,x) = \nu^{-1}e^x u(t,x + 2t - S_A \log(t/t_0))
	\]
	for all $(t,x) \in [t_0,\infty)\times [1,\infty)$.
	
	Using now \Cref{p:supersoln_r3}.(iv), we have that
	\[
		u(t,x + 2t - S_A \log(t)) \leq \bar v(t,x) e^{-x}
			\leq C (1+x)^\alpha e^{-x} 
			\qquad \text{ for all $(t,x) \geq [t_0,\infty)\times [1,\infty)$.}
	\]
	The conclusion that
	\[
		\lim_{L\to\infty} \limsup_{t\to\infty} \sup_{x \geq 2t - S_A \log(t) + L} u(t,x) = 0
	\]
	is then clear, finishing the proof.
\end{proof}

\subsubsection{The proof of \Cref{p:supersoln_r3}}

In this case (i.e., $r=3$), the nonlinearity scales as diffusion; hence, we turn to the natural diffusive self-similar variables.  For $t_0>0$ to be chosen, let 
\begin{equation*}
	\tau = \log\Big(\frac{t}{t_0}\Big),
	\qquad \text{and}\qquad
	y = \frac{x}{\sqrt{t}}.
\end{equation*}
Roughly, we choose $\overline v(t,x) \sim \zeta(\tau,y)$, and the change in variables has the advantage that the spectrum of the resulting (spatial) differential operator is discrete, as we see below.  In addition, let $\e = 1/\sqrt{t_0}$.  It is convenient, in the sequel, to choose $t_0$ large and $\e = 1/\sqrt{t_0}$ is the term which ``feels'' the effect of this.

In these new variables we proceed by choosing $\zeta$ that solves
\begin{equation}\label{e:self_similar_zeta}
	\begin{aligned}
		&\zeta_\tau + L_A\zeta
			= \e S_A e^{-\frac{\tau}{2}} \zeta_y  + \left( \fdfrac{A}{y^{2}} - \fdfrac{A}{ ( y +\mathcal{E}(\tau,y) )^{2}} \right)\zeta, \qquad &&\text{ for } \tau > 0, y > 0,\\
		&\zeta(\tau,0) = 0,   &&\text{ for } \tau>0, \\
		&\zeta(0,y) = e^{-\frac{y^2}{8}}Q, &&\text{ for } y > 0,
	\end{aligned}
\end{equation}
for $Q$ that is defined below (see \Cref{l:eigenvector_r3}) and $\mathcal{E}$ and $L_A$ defined by:
\be\label{e.mcal_E}
\begin{split}
	&\mathcal{E}(\tau,y)
		= \e e^{-\tau/2} \left( \frac{y^2}{4}
			- \alpha\log(1 + \e^{-1} e^{\tau/2} y)
			+ \log(C_0)\right)
				\quad\text{ and}\\
	&L_A\zeta = - \zeta_{yy} - \frac{y}{2} \zeta_y  + \left(\frac{A}{y^{2}} - S_A \right)\zeta.
\end{split}
\ee
We note that $L_A$ arises via the change of variables from the operator $-\Delta - (S_A/t)  + A/x^2$.

Unfortunately, $L_A$ is not self-adjoint so it is not as amenable to spectral analysis.  Let $\tilde \zeta := e^{-\alpha\tau / 2} e^{y^2/8} \zeta$ and
\[
	M_A
		:= - \partial_y^2 + \left( \frac{y^2}{16} + \frac{1}{4} + \frac{A}{y^2} - \frac{1+\alpha}{2}\right).
\]
Notice that $M_A = e^{\frac{y^2}{8}} L_A e^{-\frac{y^2}{8}} - \frac{\alpha}{2}$.  Hence, from~\eqref{e:self_similar_zeta}, we find
\begin{equation}\label{e:tildezeta_r3}
	\begin{aligned}
		&\tilde \zeta_\tau + M_A\tilde \zeta
			= \e S_A e^{-\frac{\tau}{2}} \left(\tilde \zeta_y - \frac{y}{4} \tilde \zeta\right)  + \left(\fdfrac{A}{y^{2}} - \fdfrac{A}{ ( y +\mathcal{E}(\tau,y) )^{2}}\right)\tilde \zeta, \quad &&\text{ for } \tau > 0, y > 0,\\
		&\tilde\zeta(\tau,0) = 0,   &&\text{ for } \tau>0, \\
		&\tilde\zeta(0,y) = e^\frac{y^2}{8}\zeta_0 = Q, &&\text{ for } y > 0.
	\end{aligned}
\end{equation}

Throughout, it is useful to have an upper bound on $y^{-2} - (y + \mathcal{E})^{-2}$.  We state the bound here, noting that, while it is not optimal, it is sufficient for our purposes.
\begin{lemma}\label{l:curly_E}
	If $C_0$ is sufficiently large and $\e$ is sufficiently small, then
	\[
		\left|\frac{1}{y^2} - \frac{1}{(y+\cE)^2}\right|
			\leq C\min\left\{
					\frac{4}{y^2}, 
					\e^{1/2} e^{-\tau/4}\left(\frac{1}{y} + \frac{1}{y^3}\right)
				\right\}.
	\]
\end{lemma}
\begin{proof}
	Consider first the case $y < \log(C_0) \e e^{-\tau/2}/(2\alpha)$.  Then, since $\log(1 + z ) \leq z$ for all $z>0$, we find
	\[
		\cE
			\geq \e e^{-\tau/2}
				\left( \frac{y^2}{4} - \alpha \e^{-1} e^{\tau/2} y +  \log(C_0)\right)
			\geq \e e^{-\tau/2}
				\left( \frac{y^2}{4} + \frac{\log(C_0)}{2}\right) \geq 0.
	\]
	Hence, $0 \leq y^{-2} - (y+\cE)^{-2} \leq y^{-2}$, which concludes the proof in this case.
	
	Now we consider the case where $y \geq \log(C_0) \e e^{-\tau/2}/(2\alpha)$.  First, rewrite the expression as
	\[
		\frac{1}{y^2} - \frac{1}{(y+\cE)^2}
			= \frac{\cE}{y^2(y+\cE)} + \frac{\cE}{y(y+\cE)^2}.
	\]
	If $\cE \geq 0$, then, since $\cE \leq \e e^{-\tau/2}\left(y^2 + \log(C_0)\right)$, we find
	\[
		\left|\frac{1}{y^2} - \frac{1}{(y+\cE)^2}\right|
			= \frac{\cE}{y^2(y+\cE)} + \frac{\cE}{y(y+\cE)^2}
			\leq \frac{\e e^{-\tau/2}\left(y^2 + \log(C_0)\right)}{y^3},
	\]
	from which the claim follows.
	
	If $\cE < 0$, we use that $\log$ grows sub-polynomially.  Indeed, choosing $C_0\geq e$ sufficiently large such that $\alpha \log(1 + z) \leq z^{1/4} + \log(C_0)/2$ for all $z \geq \log(C_0)/(2\alpha)$, yields
	\[\begin{split}
		\cE
			&\geq \e e^{-\tau/2}
				\left( \frac{y^2}{4} - \left(\e^{-1} e^{\tau/2} y\right)^{1/4} +  \frac{\log(C_0)}{2}\right)\\
			&\geq \e e^{-\tau/2}
				\left( \frac{y^2}{4} - \left(\frac{1}{2} + \frac{\e^{-1/2} e^{\tau/4} y^{1/2}}{2} \right) +  \frac{\log(C_0)}{2}\right)
			\geq - \frac{\e^{1/2} e^{-\tau/4} y}{2}.
	\end{split}
	\]
	which implies that $y+\cE \geq y/2$ if $\e$ is sufficiently small and that
	\[\begin{split}
		\left|\frac{1}{y^2} - \frac{1}{(y+\cE)^2}\right|
			&= -\cE \left(\frac{1}{y^2(y+\cE)} + \frac{1}{y(y+\cE)^2}\right)\\
			&\leq \frac{\e^{1/2} e^{-\tau/4} y}{2}\left( \frac{1}{y^2(y/2)} + \frac{1}{y (y/2)^2}\right)
			\leq 3\frac{\e^{1/2} e^{-\tau/4}}{y^2}.
	\end{split}\]
	The claim follows by applying Young's inequality.  This finishes the proof.
\end{proof}

We establish the spectral gap of the operator $M_A$ in $H^1_0(\R_+)$.  
\begin{lemma}\label{l:eigenvector_r3}
The function $Q(y) = Z^{-1} y^\alpha e^{- \frac{y^2}{8}}$, where $Z>0$ is such that $\|Q\|_{L^2} = 1$, solves
\[
	M_A Q = 0.
\]
Further, there exists $\lambda_A>0$, depending only on $A$, such that for all $\psi\in H^1_0(\R_+)$ such that $\inner{\psi}{Q} = 0$, we have
\[
	\inner{M_A \psi}{\psi} \geq \lambda_A\|\psi\|_2.
\]
\end{lemma}
\begin{proof}
The first claim is a straightforward computation using the fact that $\alpha(\alpha-1) = A$, by construction.  We omit it.

Since $M_A$ is a compact, self-adjoint operator, its spectrum is discrete.  Since $Q$ is positive, it is the principal eigenfunction and $0$ is the principal eigenvalue of $M_A$.  By standard theory (see, e.g., \cite{DautrayLions}), $0$ has multiplicity one.   Hence, all of the eigenvalues of $M_A$ except for $0$ are positive.  The second claim then follows via the Rayleigh quotient interpretation of the spectrum.
\end{proof}

We now analyze the long-time behavior of $\tilde \zeta$.  We begin with a preliminary estimate on the behavior of $\tilde \zeta$ near the origin.  We note that this is, roughly, an {\em a priori} estimate of $\tilde \zeta_y$ near the origin depending only on the $L^\infty$-norm of $\tilde \zeta$ in the sense that it is not (yet) a ``closed'' estimate.

\begin{lemma}\label{l:normal_bound_r3}
Fix any $T>0$ and let $\theta = \sqrt{A/(1+\alpha)}$.  There exists $\epsilon_0>0$, independent of $T$, such that if $\e \in [0,\epsilon_0]$, then, for all $(\tau,y) \in \R_+ \times [0,\theta]$,
\[
	\tilde \zeta(\tau,y)
		\leq \max\{C, 2 \|\tilde \zeta\|_{L^\infty([0,T]\times\{\theta\}}/\theta\} y.
\]
\end{lemma}
\begin{proof}
Let 
\[
	\bar N = \max\Big\{
		\max_y \frac{Q(y)}{y},
		\frac{2 \|\tilde \zeta\|_{L^\infty([0,T]\times \{\theta\})}}{\theta}
	\Big\}.
\]
We claim that $\overline \zeta(\tau,y) = \bar N y$ is a super-solution of~\eqref{e:tildezeta_r3} on $[0,T]\times [0, \theta]$; indeed,
\begin{equation}\label{e:ss_c12}
\begin{split}
	\overline \zeta_\tau &+ M_A \overline \zeta
		- \overline \zeta \left(\frac{A}{y^2} - \frac{A}{(y + \cE)^2}\right) + \eps S_A e^{-\frac{\tau}{2}} \left( \overline \zeta_y - \frac{y}{4} \overline \zeta \right)
			\\
		& = \bar N \left[y\left(  \frac{y^2}{16} + \frac{1}{4} + \frac{A}{(y + \cE)^2}
			- \frac{1+\alpha}{2}
			- \eps S_A e^{-\frac{\tau}{2}} \frac{y}{4}\right) + \lambda \eps e^{\frac{\tau}{2}}\right] \\
		& \geq \bar N\left[y\left(\frac{A}{(y+ \mathcal{E}(t,y))^2} - \frac{1}{4} - \frac{\alpha}{2}\right)\right].
\end{split}
\end{equation}
In the last line we used that, up to decreasing $\e$,
\[
	\frac{- \e S_A e^{-\tau/2} y^2}{4}
		\geq - \frac{y^3}{16} - \frac{S_A^3 \e^3 e^{-3\tau/2} 8}{27}
		\geq - \frac{y^3}{16} - S_A \e e^{-\tau/2}.
\]

We now show that the last line in~\eqref{e:ss_c12} is non-negative if $\e$ is sufficiently small.  If $y\in [0,\theta]$,
\[
	\cE(\tau,y)
		\leq \e e^{-\tau/2} \left[ \frac{y^2}{4} + \log C_0 \right]
		\leq \e  \left[\theta^2 + \log C_0 \right].
\]
Hence, we have that,
\[
	\frac{A}{(y + \cE)^2}
		\geq \frac{A}{(\theta + \e \theta^2 + \e \log C_0)^2}
		\geq \frac{A}{(\sqrt 2 \theta)^2}
		= \frac{1+\alpha}{2}
		\geq \frac{1}{4} + \frac{\alpha}{2},
\]
where the second inequality follows by choosing $\e$ sufficiently small.  
Thus, $\overline \zeta$ is a super-solution of~\eqref{e:tildezeta_r3} as claimed.

By our choice of $\bar N$, we have that $\bar \zeta \geq \tilde \zeta$ on $\{0\}\times [0,\theta]$ and $[0,T]\times \{0,\theta\}$.  The comparison principle then implies that $\bar \zeta \geq \tilde \zeta$ on $[0,T]\times [0,\theta]$, which concludes the proof.
%
%
%
%
\end{proof}

Using \Cref{l:normal_bound_r3}, we now obtain a closed bound on the $L^2$-norm of~$\tilde\zeta$.  Here, and throughout the rest of the present section, we use $\|\cdot\|$ to mean the $L^2(\R_+)$ norm for notational ease.
\begin{lemma}\label{lem:l2zetatilde}
	If $\e$ is sufficiently small, $\|\tilde \zeta(\tau)\| \leq 2$ for all $\tau\geq 0$.
\end{lemma}
\begin{proof}
	Multiplying~\eqref{e:tildezeta_r3} by $\tilde \zeta$ and integrating by parts, we find
	\[\begin{split}
		\frac{1}{2} &\frac{d}{d\tau} \|\tilde \zeta(\tau)\|^2
			+ \inner{M_A \tilde \zeta}{\tilde \zeta}
			= \e S_A e^{-\tau/2} \inner{\tilde \zeta_y - \frac{y}{4} \tilde \zeta}{\tilde \zeta}
				+ \inner{\left( \fdfrac{A}{y^2} - \fdfrac{A}{(y+ \cE)^2}\right) \tilde \zeta}{\tilde \zeta}\\
			&= - \e S_A e^{-\tau/2} \inner{y \tilde \zeta}{\tilde \zeta} + \inner{\left(\fdfrac{A}{y^2}-\fdfrac{A}{(y+ \cE)^2}\right) \tilde \zeta}{\tilde \zeta}
			\leq \inner{\left(\fdfrac{A}{y^2}-\fdfrac{A}{(y+ \cE)^2}\right)_+ \tilde \zeta}{\tilde \zeta}.
	\end{split}\]
	From \Cref{l:eigenvector_r3}, we have that $\inner{M_A \tilde \zeta}{\tilde \zeta} \geq 0$, which yields
	\[
		\frac{1}{2} \frac{d}{d\tau} \|\tilde \zeta(\tau)\|^2
			\leq  \inner{\left(\fdfrac{A}{y^2}- \fdfrac{A}{(y+ \cE)^2}\right)_+ \tilde \zeta}{\tilde \zeta}.
	\]
	
	In order to conclude, we require bounds on the right hand side above.  Fix
	\be\label{e.c7272}
		R = \min\{\theta, \e^{1/8} e^{-\tau/16}\},
	\ee
	where $\theta$ is as in \Cref{l:normal_bound_r3}.  Using \Cref{l:curly_E} and \Cref{l:normal_bound_r3}, we find
	\[\begin{split}
		&\inner{\left(\fdfrac{A}{y^2}- \fdfrac{A}{(y+ \cE)^2}\right)_+ \tilde \zeta}{\tilde \zeta}
			\leq CA\int_0^R \frac{|\tilde \zeta|^2}{y^2} dx
				+ CA\int_R^\infty \e^{1/2} e^{-\tau/4}\left(\frac{1}{y} + \frac{1}{y^3}\right) |\tilde \zeta|^2 dx\\
			&\qquad\leq CAR \left\|\tilde\zeta(\tau)/y\right\|^2_{L^\infty([0,\theta])}
				+ \frac{CA  \e^{1/2} e^{-\tau/4}}{R^3} \int_R^\infty |\tilde \zeta|^2 dx\\
			&\qquad\leq CA \e^{1/8} e^{-\tau/16}\left( \max\left\{1, \|\tilde \zeta\|^2_{L^\infty([0,\tau]\times\{\theta\})}\right\} + \|\tilde \zeta\|^2\right),	\end{split}\]
	where, in the last inequality we used the choice of $R$.
	
	
	Thus, we arrive at the differential inequality
	\begin{equation}\label{e:ss_c13}
		\frac{d}{d\tau} \|\tilde \zeta\|^2
			\leq C \e^{1/8} e^{-\tau/16} \left(1 + \|\tilde \zeta\|_{L^\infty([0,\tau]\times \{\theta\})}^2 + \|\tilde \zeta(\tau)\|^2\right).
	\end{equation}
	This is useful if we bound $\|\tilde \zeta\|_{L^\infty([0,\tau]\times \{\theta\})}$ by $\|\tilde \zeta\|$, but such a bound is not available.  However, parabolic regularity implies the following estimate that we prove in the sequel:
	\begin{equation}\label{e:L_oo_L_2}
		\|\tilde \zeta\|_{L^\infty([0,\tau]\times \{\theta\})}^2
			\leq C\Big( 1 +  \int_0^\tau \|\tilde \zeta(\tau')\|^2 d \tau'\Big).
	\end{equation}
	Before proving~\eqref{e:L_oo_L_2}, we show how to conclude the proof assuming it.
	
	From~\eqref{e:ss_c13} and~\eqref{e:L_oo_L_2}, we find
	\[
		\frac{d}{dt} \|\tilde \zeta(\tau)\|^2
			\leq C\e^{1/8} e^{-\tau/16} \left(1 + \int_0^\tau \|\tilde \zeta(\tau')\|^2 d \tau' + \|\tilde \zeta(\tau)\|^2\right).
	\]
	Let $\tau_0 = \sup\{\tau: \|\tilde \zeta(\tau')\|^2 \leq 2$ for all $\tau'\in[0,\tau]\}$.  If $\tau_0 = \infty$, then we are finished.  Hence, suppose that $\tau_0$ is finite.  Then, integrating the above between $0$ and $\tau_0$, we see that
	\[\begin{split}
		1
			&= \|\tilde \zeta(\tau_0)\|^2 - \|\tilde\zeta(0)\|^2
			\leq C \e^{1/8} \int_0^{\tau_0} e^{-\tau/16} \left( 1 + \int_0^{\tau} 2 d \tau' + 2\right) d \tau\\
			&\leq C \e^{1/8} \left( 8 + 2 \int_0^{\tau_0} \tau e^{-\tau/16} d\tau + 16\right)
			\leq (8 + 2 \cdot 16^2 + 16)C \e^{1/8}.
	\end{split}\]
	This yields a contradiction if $\e$ is sufficiently small.  Hence, the proof is concluded once~\eqref{e:L_oo_L_2} is established.
	
	We now establish~\eqref{e:L_oo_L_2}.  Observe that
	\begin{equation}\label{e:ss_c14}
		\|\tilde \zeta\|_{L^\infty([0,\tau]\times \{\theta\})}
			\leq \|\tilde \zeta\|_{L^\infty([0,2]\times \{\theta\})}
				+ \sup_{\tau_0 \in [2,\max\{2,\tau\}]} \|\tilde \zeta\|_{L^\infty([\tau_0-1,\tau_0]\times \{\theta\})}.
	\end{equation}
	We bound each term on the right separately.
	
	The first term on the right in~\eqref{e:ss_c14} is bounded using the comparison principle.  Indeed, let $\overline \zeta(\tau,y) = N e^{R \tau} y$, where $N = \max Q(y)/y$ and $R = (1+\alpha)/2$ and notice that
	\[\begin{split}
		\overline \zeta_\tau &+ M_A \overline \zeta
		- \overline \zeta \Big(\fdfrac{A}{y^2} - \fdfrac{A}{(y + \cE)^2}\Big) + \eps S_A e^{-\frac{\tau}{2}} \Big( \overline \zeta_y - \frac{y}{4} \overline \zeta \Big)\\
		&= N \Big[y\Big(  \fdfrac{y^2}{16} + \fdfrac{1}{4} + \fdfrac{A}{(y + \cE)^2}
			- \eps S_A e^{-\frac{\tau}{2}} \fdfrac{y}{4}\Big) + \lambda \eps e^{\frac{\tau}{2}}\Big].
	\end{split}\]
	Using Young's inequality and decreasing $\e$, the right hand side is non-negative.  It follows that $\overline\zeta$ is a supersolution of~\eqref{e:tildezeta_r3} and, hence, that $\tilde \zeta \leq \overline \zeta$.  We conclude that
	\begin{equation}\label{e:ss_c15}
		\|\tilde \zeta\|_{L^\infty([0,2]\times \{\theta\})}
			\leq \|\overline \zeta\|_{L^\infty([0,2]\times \{\theta\})}
			= N e^{1+\alpha} \theta.
	\end{equation}
	
	Next, we consider the second term on the right in~\eqref{e:ss_c14}.  We may clearly assume $\tau \geq 2$. 
	Here, we use parabolic regularity estimates as follows.  For any $\tau_0 \in [2,\tau]$, standard interior parabolic regularity estimates in Sobolev spaces imply that
	\[\begin{split}
			C\|\tilde \zeta\|_{H^2_{\rm para} ([\tau_0-1,\tau_0]\times [\theta/2,3\theta/2])}^2
			\leq C \|\tilde \zeta\|_{L^2([\tau_0-2,\tau_0]\times [\theta/4,2\theta])}^2,
	\end{split}\]
	where we use the notation $H^2_{\rm para} = \{u \in L^2 : Du, D^2u, u_t \in L^2\}$ for the standard second order parabolic Sobolev space.  By the Sobolev embedding theorem, we find
	\[
		\|\tilde \zeta\|_{L^\infty([\tau_0-1,\tau_0]\times [\theta/2,3\theta/2])}^2
			\leq C \|\tilde \zeta\|_{L^2([\tau_0-2,\tau_0]\times [\theta/4,2\theta])}^2
			\leq C \int_0^\tau \|\tilde \zeta(\tau')\|^2 d\tau'.
	\]
	Since this holds for all $\tau_0\in [2,\tau]$, the combination of this,~\eqref{e:ss_c15}, and~\eqref{e:ss_c14} establishes~\eqref{e:L_oo_L_2}, which completes the proof.
%
\end{proof}

In fact, relying on parabolic regularity theory (see the proof above), we obtain a uniform $L^\infty$ bound on $\tilde \zeta$ on $[0,\infty)\times [\theta/2,\infty)$.  Pairing this with \Cref{l:normal_bound_r3}, we obtain the following lemma.
\begin{lemma}\label{lem:linftyzetatilde}
	If $\e$ is sufficiently small, there exists a constant $C$ such that
	\[
		\| \max\{1,y^{-1}\}\, \tilde \zeta  \|_{L^\infty(\R_+ \times \R_+)}
			\leq C.
	\]
\end{lemma}

We require one final estimate on $\tilde \zeta$, which shows that the $Q$ part of $\tilde \zeta$ dominates the long time behavior.  We establish that here.

\begin{lemma}\label{lem:Q_tildezeta}
	If $\e$ is sufficiently small, there is $\beta>0$ such that, for all $y \geq 0$,
	\[
		|\tilde \zeta(\tau,y) - \tilde \zeta(0,y)|
			\leq C \e^{1/8} \left(Q(y) + e^{-\beta \tau}\right).
	\]
\end{lemma}
\begin{proof}
We aim to use the spectrum of $M_A$ in order to conclude.  Decompose $\tilde \zeta$ into its $Q$ component and its orthogonal component, letting
\[
	\phi = \inner{\tilde \zeta}{Q}
		\quad \text{and} \quad
	\tilde \zeta =  \phi Q + \psi.
\]
To estimate the $Q$ component, multiply~\eqref{e:tildezeta_r3} by $Q$ and integrate by parts to find
\be\label{e.c7273}
\begin{split}
	|\phi'(\tau)|
		&= A \int_{\R^+} \left| \frac{1}{y^2} - \frac{1}{(y+\cE)^2}\right| \tilde \zeta(y) Q(y) \, dy
			+ \left|\e S_A e^{-\tau/2} \langle Q_y + \frac{y}{4} Q\vert \tilde \zeta\rangle\right|,\\
		&\leq A \int_{\R^+} \left|\frac{1}{y^2} - \frac{1}{(y+\cE)^2}\right| \tilde \zeta(y) Q(y) \, dy
			+ \e S_A e^{-\tau/2} \Vert Q_y + \frac{y}{4} Q \Vert \Vert  \tilde \zeta (\tau,\cdot)\Vert,
%
\end{split}
\ee
The estimate of the first term goes exactly as in~\Cref{lem:l2zetatilde}, where we applied \Cref{l:curly_E} to conclude.  Hence, we omit the details and assert that
\[
	\int_{\R^+} \left|\frac{1}{y^2} - \frac{1}{(y+\cE)^2}\right| \tilde \zeta(y) Q(y) \, dy
		\leq \e^{1/8} e^{-\tau/16} \left(\|\max\{1,y^{-1} \tilde \zeta\|_{L^\infty(\R_+ \times \R_+)} + \|\tilde \zeta\|\right).
\]
Recalling \Cref{lem:l2zetatilde} and \Cref{lem:linftyzetatilde} and combining with~\eqref{e.c7273} yields
\[
	|\phi'(\tau)|
		\leq C \e^{1/8} e^{-\tau/16}.
\]
Integrating this differential inequality, we deduce that
\begin{equation*}
\vert \phi(\tau) - \phi(0) \vert \leq C \e^\frac18.
\end{equation*}

We now consider $\psi$.  Since $\psi \perp Q$, we have $\inner{M_A \tilde \zeta}{\psi} = \inner{M_A \zeta}{\zeta} \geq \lambda_A \|\psi\|^2$.  Hence, multiplying~\eqref{e:tildezeta_r3} by $\psi$ and integrating, we find
\begin{align*}
&\frac12 \frac{d}{dt} \Vert \psi \Vert^2 + \inner{M_A \psi}{\psi}
		= \inner{\left(\frac{A}{y^2} - \frac{A}{(y + \cE)^2}\right) \tilde \zeta}{ \psi} + S_A \eps e^{-\frac{\tau}{2}} \inner{\phi Q_y + \psi_y  - \frac{y}{4} \left( \phi Q + \psi\right)}{\psi}.
\end{align*}
The second inner product is estimating using that $\inner{\psi_y}{\psi} = 0$, $\inner{y\tilde\zeta}{\psi} \geq 0$, and the Cauchy-Schwarz inequality.  On the other hand, arguing as above using \Cref{l:curly_E}, we bound the first time by
\begin{equation*}
	\inner{\left(\frac{A}{y^2} - \frac{A}{(y + \cE)^2}\right) \tilde \zeta}{\psi}
		=  \inner{\left(\frac{A}{y^2} - \frac{A}{(y + \cE)^2}\right) \tilde \zeta}{\tilde \zeta - \phi Q}
		\leq 
			C \e^{1/8} e^{-\tau/16}.
\end{equation*}
Using all ingredients above and recalling that $\inner{M_A \psi}{\psi} \geq \lambda_A \|\psi\|^2$, we obtain the differential inequality
\[
	\frac{1}{2} \frac{d}{dt} \|\psi\|^2
		+ \lambda_A \|\psi\|^2
		\leq C \e^{1/8} e^{-\tau/16}
			+ C \e e^{-\tau/2} \|\psi\|.
\]
Solving this differential inequality, there is $\beta>0$ such that
\[
	\|\psi\| \leq C \e^{1/8} e^{-\beta \tau}.
\]
The claim then follows by using parabolic regularity theory in order to upgrade the $L^2$ convergence to $L^\infty$ convergence.
\end{proof}

We are now able to conclude the proof of \Cref{p:supersoln_r3} using all ingredients above.
\begin{proof}[Proof of \Cref{p:supersoln_r3}]
Choose $t_0, C_0>0$ sufficiently large such that all above lemmas hold.  Fix $M>0$ to be determined.  Let $\tilde \zeta$ be the solution of~\eqref{e:tildezeta_r3} and define, for $t \geq t_0$ and $x \geq 0$,
\[
	\overline v(t,x)
		= M\left(\frac{t}{t_0}\right)^{\frac{\alpha}{2}} e^{- \frac{x^2}{8 t}} \tilde \zeta(\log(t/t_0), x/\sqrt t).
\]
By construction, we have that \Cref{p:supersoln_r3}.(i) holds.  

Next we examine the ordering of $v$ and $\overline v$ at $t= t_0$.  From \Cref{l.heat_bounds}, we have, for all $x \geq 1$
\[
	v(t_0,x)
		\leq \frac{C\sqrt{t_0}}{x+2t_0 + \sqrt{t_0}} e^{t_0 + x - \frac{(x + 2t_0)^2}{4t_0}}
		\leq \frac{C\sqrt{t_0}}{x+2t_0} e^{t_0 - \frac{x^2}{4t_0}}.
\]
On the other hand,
\[
	\overline v(t_0,x)
		= \frac{M}{Z} \left(\frac{x}{\sqrt{t_0}}\right)^\alpha e^{- \frac{x^2}{4t_0}}.
\]
It is clear that, up to increasing $M$, $\overline v(t_0,x) \geq v(t_0,x)$ for all $x\geq 1$.  This yields \Cref{p:supersoln_r3}.(ii).

The proof of \Cref{p:supersoln_r3}.(iii) follows similarly, using \Cref{lem:Q_tildezeta} in addition.  We omit the details.

Finally, we consider \Cref{p:supersoln_r3}.(iv).  If $x\geq \sqrt t$, we have
\[
	\overline v(t,x)
		\leq M \left(\frac{t}{t_0}\right)^{\frac{\alpha}{2}} \|\tilde \zeta\|_{L^\infty(\R_+\times\R_+)},
\]
which is less than $(1+x)^\alpha$, finishing the proof in this case.  If $x \leq \sqrt t$, we apply \Cref{lem:Q_tildezeta} to find
\[
	\overline v(t,x)
		\leq M \left(\frac{t}{t_0}\right)^{\frac{\alpha}{2}} \left(Q(x/\sqrt t) + C \e^{1/8} e^{-\beta \tau}\right)
		\leq M \left(\frac{t}{t_0}\right)^{\frac{\alpha}{2}} \left(\frac{1}{Z} \left(\frac{x}{\sqrt t}\right)^\alpha + C \e^{1/8} e^{-\beta\tau}\right).
\]
The conclusion follows from the above.  The proof is finished.
\end{proof}

\section{The case $r\in (1,3)$: \Cref{thm:main_delay}.(iii)}\label{s.lw3}

To begin, we give a brief (heuristic) description of the proof, and, along the way, define the key concepts that we rely on in the sequel.  Two important constants in our analysis are
\[
	\gamma = \frac{2}{1+r}
		\quad\text{ and } \quad
	\beta = \frac{3-r}{1+r}.
\]

We begin by changing variables to the moving frame and removing an exponential factor; that is, we let $u(t,x+2t - s(t)) = \nu e^{-x} v(t,x)$. Then
\be\label{e:1141}
	v_t + \dot s(t)(v_x - v)
		= v_{xx} - A (x + \log(1/v))^{-(r-1)} v.
\ee
From~\cite{BouinHendersonRyzhik}, we know understand that the correct length scale to look on is $x \sim t^\gamma$.  Hence, let
\be
	\varphi(\tau, y)
		= -\frac{1}{\tau} \log v(\tau^{1/\beta}, y \tau^{\gamma/\beta}).
\ee
Note very importantly the large deviations flavour of this ansatz: this is where the mixing of scales appears and where $\gamma > \frac12$ plays a role. The correct length scale to look is larger than the diffusive one. Nevertheless, this can be connected to the standard diffusive change of variables used in the cases $r\geq 3$ in the following way: an equivalent (but computationally more complicated) change of variables is $\tau = \frac{1}{\beta} \left[ (1+t)^\beta - 1\right]$ and $y = \frac{x}{(1+t)^\gamma}$, which yields the diffusive variables in the limit $r\to3$.

Recall from~\cite{BouinHendersonRyzhik} that the delay should be $O(t^\beta)$.  Set (again, heuristically) $s(t)= \theta t^\beta$, where the goal is to determine $\theta$ so that $w = O(1)$ near the origin (since we expect $u$ to be $O(1)$ near the front). This requires $\varphi(\tau, 0) = O(1/\tau)$ and gives the following
\be\label{e:phi_sr}
	\beta\tau \varphi_\tau - \gamma y \varphi_y + \beta \varphi  + \theta \beta \left( \tau^{- \frac{r-1}{3-r}} \varphi_y + 1\right)  
		= \tau^{-1} \varphi_{yy} - |\varphi_y|^2 + A \left(y + \tau^{-\frac{r-1}{3-r}} \varphi\right)^{-(r-1)}.
\ee
Formally taking $\tau \to \infty$ (and assuming that $\tau \varphi_\tau \to 0$ since we expect equilibrium dynamics), we obtain the limit equation
\be\label{e:c11311}
	\begin{aligned}
		&| \varphi_y|^2 - \gamma y \varphi_y - (A y^{1-r} - \beta(\theta + \varphi))= 0
			\qquad &\text{ in } (0,\infty),\\
		&\varphi(0) = 0.
	\end{aligned}
\ee
We now explain how to guess the correct value of $\theta$. This comes by comparing the asymptotics of the solutions of such an ODE to the expected asymptotics of $w$.

One solution of this quadratic polynomial in \eqref{e:c11311} is as follows. Let
\be\label{e:Gamma}
	\Gamma(y)
		= \frac{\gamma^2 y^2}{4\beta} + \frac{A y^{1-r}}{\beta}
			\quad\text{ for all } y >0,
\ee
and define $\ophim_\theta$ (subscript often omitted later for legibility) via
\be\label{e:phi}
\begin{split}
	&\ophim_y(y)
		=  \frac{ \gamma y}{2} - \sqrt{ \beta (\Gamma - \phi)}, \qquad y>0,\\
	&\ophim(0) = \theta.
\end{split}
\ee
Observe that with such a definition, $\ophim - \theta$ solves \eqref{e:c11311}.  We note a subtle notational choice here: $\varphi$ refers to the (expected) limiting solution, while $\phi$ refers to any of the solutions to the shifted family of initial value problems~\eqref{e:phi}.

The global existence of such a $\ophim$ on $\R^+$ as a function of the initial data $\theta$, is discussed in \ref{s:phi}. If we expect convergence of $u$ to a traveling wave in the moving frame, it is then natural in view of \Cref{p:decayTW} to expect that $v(t,x) \sim \exp\Big\{ \frac{2A^{\frac12}}{3-r} x^{\frac{3-r}{2}} \Big\}$ 
close to zero. This fits exactly with the asymptotics of $\ophim$ near $y=0$ for any compatible value of $\theta$.  We make two observations from this.  First, since this works for all $\theta$, $\theta$ cannot be defined at this stage. Second, we arrived at~\eqref{e:phi} from~\eqref{e:c11311} via the quadratic formula, which involves choosing a root.  Had we chosen the {\em other} root, the traveling wave asymptotics would {\em not} hold, allowing us to conclude that we have chosen the correct root for small $y$.

However, we expect $v(t,x) \sim e^{-\frac{x^2}{4t}}$ very far ahead of the front.  This corresponds to $\varphi \sim y^2/4$ when $y \gg 1$. 
Unfortunately, if $\theta\gg1$ then $\phi$ cannot be extended as a solution to large $y$, and, even for those $\theta$ for which it can, $\phi$ does not grow quadratically. Thus this $\ophim$ cannot the expected $\varphi$ when $\tau$ goes to infinity. We now explain how to solve this issue and this will give the value of $\theta$ to be chosen. 

Using $\ophim$, we define
\be\label{e:critical_theta}
	\Theta
		= \sup \Big\{ \theta \in \R : \ophim_\theta < \Gamma(y) \text{ for all } y \geq 0\Big\},
\ee
where the curve $\Gamma$ is given by~\eqref{e:Gamma}.  It is clear that if $\theta < \Theta$ then $\ophim$ exists on $[0,\infty)$. The positivity and finiteness of $\Theta$ is shown in \Cref{s:phi}.

By a continuity argument, if $\theta = \Theta$, $\ophim$ touches the curve $\Gamma$ ``tangentially''.  In the sequel, we show that there is only one touching point $\bar y$. Thus, when $\theta = \Theta$, we construct the $C^1$ globally defined solution $\ophic$ which is equal to $\ophim$ to the left of $\bar y$ and solves
\be\label{e:Phi}
\begin{split}
	&\ophic_y(y)
		= \frac{\gamma}{2} y + \sqrt{ \beta (\Gamma - \ophic)}
			\qquad \text{ in } (\bar y, \infty).
\end{split}
\ee
With this definition, $\ophic$ solves \eqref{e:c11311}, $\ophic \sim \Theta - \frac{2\sqrt A}{3-r} y^\frac{3-r}{2}$ when $y\sim 0$, and $\ophic \sim y^2/4$ when $y\gg1$.  Hence $\ophic$ is the solution of~\eqref{e:c11311} that should arise from $\varphi$ when taking $\tau \to \infty$, above. Making this precise choice $\theta = \Theta$ is the only way to make this happen.

It is now heuristically clear that $\theta = \Theta$ is the correct choice of the shift.  Before continuing, we simply note that $\Theta = \Theta_r A^\gamma$ for $\Theta_r$ independent of $A$ is easy to see by a simple scaling argument.  We stress, though, that $A$ may not be scaled out of the original equation~\eqref{e:main}; it is a feature of our reduced characterization of the delay coefficient above that the scaling in $A$ may be removed.

The remaining part of this section will construct rigorously $\ophim$ and $\ophic$ and provide some qualitative properties that are needed later on to construct sub- and super- solutions in the shifted frame.


\subsection{Construction of $\Theta$ and behaviors of $\ophim$ and $\ophic$}\label{s:phi}

In this section, we let $\bar y = (1+r)^\gamma A^{\frac{\gamma}{2}}$, which plays a special role in the analysis.  Indeed, it is where $\phi_\Theta$ and $\Gamma$ touch (see \Cref{p:phi}).

\subsubsection{Existence and qualitative properties of $\ophim$}


For the proof of the upper bound on the front location, it is enough to work only with $\phi$.  Despite the fact that $\Phi$ is expected to provide the asymptotics of $\varphi$ as $\tau \to \infty$, it is somewhat easier to work with $\phi$ as its growth as $y\to\infty$ is slower, and thus, it is easier to ``fit'' a supersolution built from $e^{-\phi}$ over $u$.  The requisite bounds are below.  Hence we obtain bounds of $\phi$ on all of $(0,\infty)$ instead of simply on $(0,\bar y)$.

\begin{prop}\label{p:phi}
%
%
The constant $\Theta$ is positive and finite.  The solution $\phi = \phi_\Theta$ is defined on all of $\R_+$, is strictly less than $\Gamma$ on $\R_+\setminus \{\overline y\}$, and, at $\overline y$, touches $\Gamma$ tangentially; that is, $\Gamma(\overline y) = \phi(\overline y)$ and $\phi'(\overline y) = \Gamma'(\overline y)$.  In addition, $\phi$ satisfies:
%
%
\begin{enumerate}
\item $C^{-1}
		\leq \ophim
		\leq Cy^{\frac{1+r}{2}}\1_{[C^{-1},\infty]}
			+ C\1_{[0,C^{-1}]}$
		
\item $\displaystyle
		- C y^\frac{1-r}{2} \1_{[0,C]}
			\leq \phi_y
			\leq - C^{-1} y^\frac{1-r}{2} \1_{[0,C^{-1}]}$,
\item $\displaystyle
	 - C\1_{[C^{-1},\infty]}
		+ C^{-1} y^{-\frac{1+r}{2}} \1_{[0,C^{-1}]} 
			\leq \phi_{yy}
			\leq  Cy^{-\frac{1+r}{2}}$.
\end{enumerate}
\end{prop}
\begin{remark}\label{r:Theta}
\Cref{p:phi} yields a second characterization of $\Theta$ is that $\Theta =\phi(0)$ where $\phi$ solves the terminal value problem
\[
	\begin{split}
		&\ophim'(y)
			=  \frac{\gamma}{2} y - \sqrt{ \beta (\Gamma - \phi)}, \qquad y \in (0, \bar y),\\
		&\ophim(\bar y) = \Gamma(\bar y).
	\end{split}
\]
The uniqueness of such a solution is somewhat subtle but follows from the fact that such a solution necessarily has $\phi'(\overline y) > 0$.  We omit the details.
\end{remark}
\begin{proof}[Proof of \Cref{p:phi}]
For any $\theta \in \R$, there exists a solution $\phi = \phi_\theta$ of the Cauchy problem~\eqref{e:phi} by Carath\'eodory's theorem on some interval $[0,y_\theta]$.  We note that the standard Cauchy-Lipschitz theorem does not apply at $y=0$ due to the singularity in $\Gamma$. 
It is easy to check that this is unique due to the fact that the singularity at $y=0$ is integrable; indeed, the right hand side of~\eqref{e:phi} is $\sim \sqrt -A y^\frac{1-r}{2}$ near $y=0$.  Finally, we note that, by the standard Cauchy-Lipschitz theorem, $y_\theta$ can be increased for as long as  $\phi < \Gamma$.

%


We now justify that $\Theta$ 
is positive and finite. First, if $\theta = 0$, the phase portrait shows that $\ophim$ is then decreasing, thus always negative. As a result, $\ophim$ is defined on $\R^+$.  By continuity with respect to initial conditions, if $0 < \theta \ll 1$, then $\ophim$ eventually becomes negative and, thus, exists for all $y$.  We conclude that $\Theta > 0$.

Assume by contradiction that $\Theta$ is infinite. This means that for all $\theta \in \R$, the solution of the ODE is globally defined.  Let $\Psi: \R \to C^1([0,\infty))$ be defined by $\Psi(\theta) = \ophim_\theta$; that is, the solution of~\eqref{e:phi} with $\ophim_\theta(0) = \theta$.  Since solutions cannot cross, $\Psi$ is increasing. Moreover, it is then direct from the definition of $\ophim_\theta$ that for all $y \in \R^+$, one has $\Gamma - \phi_\theta \geq 0$. As a consequence, the sequence of functions, defined on $\R_+^*$, $\theta \mapsto \phi_\theta(\cdot)$ is convergent. However, $g :=\partial_\theta \phi_\theta$ satisfies 
\begin{align*}
	&g_y(y) = \frac{2\beta^\frac12}{\sqrt{ \Gamma - \phi }} g(y)
		\qquad \text{and}\qquad
	g(0) = 1,
\end{align*}
so that
\[
	g(y) = \exp\left\{\int_{0}^y \frac{2\beta^\frac12}{\sqrt{ \Gamma(z) - \phi(z) }}  \, dz  \right\}
		\geq 1.
\]
This is a contradiction with the convergence of the sequence.

In order to construct $\phi_\Theta$, we use a limiting procedure as follows.  If $\theta < \Theta$, then $\phi_\theta<\Gamma$ by the definition of $\Theta$.  The equation yields that $\phi_\theta$ is $C_{\rm loc}^{1,1/2}$ with bounds that are uniform in $\theta$.  We may thus take $\theta \nearrow \Theta$ to obtain a function $\phi_\Theta$ that is $C_{\rm loc}^{1,1/2}$.  Using the expansion $\phi_\theta \sim \theta - \frac{2}{3-r} y^\frac{3-r}{2}$, we see that $\phi_\Theta(0) = \Theta$ and $\phi_\Theta$ is continuous up to $0$.  Hence, $\phi_\Theta$ satisfies~\eqref{e:phi}.

If $\phi_\Theta$ does not touch $\Gamma$ then we can further increase $\theta$, contradicting the choice of $\Theta$.  By continuity, $\phi_\Theta$ must touch $\Gamma$ tangentially, in which case the two curves necessarily touch at $\overline y$ since this is the only solution to $\phi_\Theta'(y) = \Gamma'(y)$.

Finally, we analyze bounds on $\phi'' = \phi_\Theta''$.  Those bounds away from $\bar y$ are simple to establish and so we omit them.  We focus instead on a uniform bound on $\phi''$ away from $y=0$; indeed, we establish a bound on $(\bar y/2, \infty)$.  To do this, we work with $\phi = \phi_\theta$ for $\theta < \Theta$.  A limiting argument then yields the bound for $\phi_\Theta$.

Since $\theta < \Theta$, $\phi$ is smooth and we differentiate~\eqref{e:phi} to find
\be\label{e:opsip}
	\begin{aligned}
		&\ophim''
		= \frac{\gamma - \beta}{2} - \frac{\beta^\frac12}{2} \frac{\Gamma'- \frac{\gamma}{2}y}{\sqrt{ \Gamma - \ophim}}.
	\end{aligned}
\ee
Clearly, we need only consider the case when $\Gamma - \phi \ll 1$ and we need only examine the extrema of the second term,  
which occur when
\[
	2 \left(\Gamma - \phi\right) \left( \Gamma'' - \frac{\gamma}{2}\right)
		= \left( \Gamma' - \frac{\gamma y}{2}\right)
			\left(\Gamma' - \phi'\right).
\]
Re-writing this using the form of $\phi'$, letting $Z = \Gamma - \phi$, and letting $Y = \Gamma' - \gamma y/2$, we find, at the extremum,
\[
	Z
		= \frac{Y}{2\left( \Gamma'' - \frac{\gamma}{2}\right)} \left( Y +  \frac{\sqrt\beta}{2} \sqrt Z\right).
\]
As we are only considering the domain $(\bar y/2,\infty)$, $\Gamma'' - \gamma/2$ is bounded above and below by a constant.  In addition, we are considering only the case where $Z \ll 1$.  Hence, the above can only hold for $Y$ satisfying
\be\label{e.c842}
	|Y| \leq C \sqrt Z
\ee
for a constant $C$ depending only $r$.  We conclude, from~\eqref{e.c842} that
\[
	\left|\frac{\Gamma' - \frac{\gamma}{2} y}{\sqrt{\Gamma - \phi}}\right|
		\leq C.
\]
The desired bound on $\ophim''$ follows from this and~\eqref{e:opsip}, thus concluding the proof.
\end{proof}

%
%
%
%
%

\subsubsection{The construction of $\ophic$}

In order to establish the lower bound on the front, we construct $\Phi$ and establish the requisite bounds.

\begin{prop}\label{p:Phi}
The function $\ophic: [0,\infty) \to \R$, defined by $\ophic = \phi_\Theta$ on $[0,\bar y]$ and solving
\begin{equation}\label{eq:Phi}
	\ophic' = \frac{\gamma}{2} y +\sqrt{ \beta (\Gamma - \ophic)}
				\qquad  \text{on } (\bar y,\infty)
\end{equation}
solves~\eqref{e:c11311}, is in $C^1_{\rm loc}(0,\infty)\cap W^{2,\infty}_{\rm loc}(0,\infty)$, and satisfies the bounds
	\[
		- C y^{-\frac{r-1}{2}}
			\leq \ophic'(y)
			\leq Cy \1_{[C^{-1},\infty)} - C^{-1} y^{-\frac{r-1}{2}}
		\quad\text{ and }\quad
		0 \leq \ophic''(y)
			\leq C\left( 1 + y^{- \frac{1+r}{2}}\right). 
	\]
	In addition, for $y \geq \bar y$, we have
	\[
		\frac{y^2}{4}
			\leq \Phi(y)
			\leq \frac{y^2}{4} + C.
	\]
\end{prop}

\begin{proof}
The construction of $\Phi$ on $(\bar y, \infty)$ can be done using standard methods\footnote{The fact that the right hand side of~\eqref{eq:Phi} is not Lipschitz in $\Phi$ is not used in the existence portion of standard well-posedness proofs; hence, the standard fixed point proof suffices for our setting.  While we do not require uniqueness of the solution on $(\bar y,\infty)$, this can be established, although with slightly more difficulty than the usual method.  Hence, we omit it.}.  We note that $\mathcal{C}^2$ bounds on $\Phi$ away from $0$ and $\infty$ can be established exactly as in \Cref{p:phi}.  The convexity follows from the fact that
\be\label{e:Phi_yy}
	\begin{aligned}
		&\ophic_\theta''
		= \frac{\gamma - \beta}{2} - \frac{\beta^\frac12}{2} \frac{\Gamma'- \frac{\gamma}{2}y}{\sqrt{ \Gamma - \ophic_\theta}}&\text{ in } (0,\bar y),\\
		&\ophic_\theta''
		= \frac{\gamma- \beta}{2} +\frac{\beta^\frac12}{2} \frac{\Gamma'- \frac{\gamma}{2}y}{\sqrt{ \Gamma - \ophic_\theta}}&\text{ in } (\bar y,\infty).
	\end{aligned}
\ee
and that $\Gamma' - \gamma y/2$ is negative in $(0,\bar y)$ and positive in $(\bar y, \infty)$.

The bounds on $\Phi'$ are clear from~\eqref{eq:Phi}, and the upper bound on $\Phi''$ near $y=0$ follows from \Cref{p:phi}.  Hence, we need only establish behavior of $\Phi$ and $\Phi''$ when $y \gg 1$.

To this end, observe that $\Phi_{\text{inf}} (y) \equiv y^2/4$ is a sub-solution of \eqref{eq:Phi} on $(\overline{y} , \infty)$. Indeed, $\ophic(\overline y) = \Gamma(\overline y) \geq \Phi_{\text{inf}} (\overline y)$, and since $\sqrt{\gamma^2-\beta} = 1-\gamma$,
\begin{align*}
\Phi_{\text{inf}}'(y)
	= \frac{y}{2} = \left(\gamma + \sqrt{\gamma^2-\beta} \right)\frac{y}2
	\leq \frac{\gamma}{2} y +\sqrt{ \beta \left(\frac{\gamma^2 y^2}{4\beta}  -\frac{y^2}{4}\right)} 
	&\leq \frac{\gamma}{2} y +\sqrt{ \beta (\Gamma - \Phi_{\text{inf}})}.
\end{align*}
Hence, $\Phi_{\inf} \leq \Phi$, by the comparison principle.

Observe that $\Phi_{\text{sup}} (y) \equiv y^2/4 + \frac{A \overline{y}^{1-r}}{\beta}$ is a super-solution of \eqref{eq:Phi} on $(\overline{y} , \infty)$. Indeed, $\ophic(\overline y) = \Gamma(\overline y) \leq \Phi_{\text{sup}} (\overline y)$ for $d$ large, and since $\sqrt{\gamma^2-\beta} = 1-\gamma$,
\begin{align*}
\Phi_{\text{inf}}'(y)
	&= \frac{y}{2} = \left(\gamma + \sqrt{\gamma^2-\beta} \right)\frac{y}2
	= \frac{\gamma}{2} y +\sqrt{ \beta \left(\Gamma(y) - \frac{A y^{1-r}}{\beta}  -\frac{y^2}{4}\right)} \\
	&\geq \frac{\gamma}{2} y +\sqrt{ \beta \left(\Gamma(y) - \frac{A \overline{y}^{1-r}}{\beta}  -\frac{y^2}{4}\right)} 
	\geq \frac{\gamma}{2} y +\sqrt{ \beta (\Gamma - \Phi_{\text{inf}})}.
\end{align*}
Hence, $\Phi \leq \Phi_{\sup}$, by the comparison principle.  Using these upper bounds in~\eqref{e:Phi_yy} yields the bounds on $\Phi''$ when $y\gg 1$.  This concludes the proof.
\end{proof}

\subsection{An upper bound on the front location}

The first step to proving the upper bound in \Cref{thm:main_delay}.(iii) is to build a supersolution of~\eqref{e:main} with $\ophim$, which was constructed in \Cref{s:phi}.  We work in the shifted frame with an increasing delay $s(t)$ to be determined.  Writing
\[
	u(t,x+2t - s(t))
		= \nu e^{-x} w(t,x),
\]
we see that
\[
	w_t + \dot s (w_x-w)
		= w_{xx} - A w ( x + \log(1/w))^{-(r-1)}.
\]

We use the natural change of variables discussed above ($\tau = t^\beta$ and $y = \frac{x}{t^\gamma}$), and define
\be\label{e.c844}
	v(\tau,y)
		= w( \tau^{1/\beta}, y \tau^{\gamma/\beta})
		\qquad\text{and}\qquad
	S(\tau) = s(\tau^{1/\beta}),
\ee
The new function $v$ satisfies
\begin{equation}\label{e:v_supersolution}
	\beta v_\tau		= \tau^{-2}  v_{yy}
			+ \beta\dot S v 
			- \beta\dot S \tau^{-\frac{\gamma}{\beta}} v_y
			+ \gamma \tau^{-1} y v_y
			- A( y +\tau^{-\frac{\gamma}{\beta}} \log(1/v))^{-(r-1)} v.
\end{equation}
Finding a supersolution of~\eqref{e:v_supersolution} yields an upper bound of $u$. We set
\begin{equation*}
S(\tau) = \Theta \tau - R(\tau),
\end{equation*}
with $\Theta$ defined in \eqref{e:critical_theta} and $R$ chosen in \Cref{l:supersolution_small_r}.

\begin{lemma}\label{l:supersolution_small_r}
Fix $\ophim$ solving~\eqref{e:phi} with $\theta = \Theta$ and let $R_0$ be any constant.  There exists a $C^1_{\rm loc}$ increasing function $R: (0,\infty)\to \R$ that is independent of $R_0$, such that $R(\tau) /\tau \to 0$ as $\tau \to\infty$, $R \geq 0$ on $[1,\infty)$, and the function
\[
	\overline v(\tau,y)
		= \exp\left\{R_0 + \tau (\Theta -  \ophim(y))\right\}
\]
is a supersolution of~\eqref{e:v_supersolution} on the domain $\{(\tau,y): y \geq 0, \tau \geq 1\}$.
\end{lemma}
\begin{proof}
We leave $R$ arbitrary and choose it in the sequel.  Using~\eqref{e:phi}, we compute that
\[
	\begin{split}
		\beta \overline v_\tau
			&-\tau^{-2}  \overline v_{yy}
			- \beta\dot S \overline v 
			+ \beta\dot S \tau^{-\frac{\gamma}{\beta}} \overline v_y
			- \gamma \tau^{-1} y \overline v_y
			+ A( y +\tau^{-\frac{\gamma}{\beta}} \log(1/\overline v))^{-(r-1)} \overline v\\
			&= \left(
				\beta \Theta
				- \beta \ophim
				- |\ophim_y|^2 + \tau^{-1} \ophim_{yy}
				- \beta \dot S
				- \beta \dot S\tau^{1 - \frac{\gamma}{\beta}}\ophim_y
				+ \gamma y \ophim_y + A( y +\tau^{-\frac{\gamma}{\beta}} \log(1/\overline v))^{-(r-1)}
			\right)  \overline v\\
			&= \left(
				\beta \dot R
				+ \tau^{-1}  \ophim_{yy}
				- \beta \dot S\tau^{1 - \frac{\gamma}{\beta}}\ophim_y
				+ A( y +\tau^{-\frac{\gamma}{\beta}} (-R_0 - \tau \Theta + \tau \ophim))^{-(r-1)} - Ay^{-(r-1)}
			\right)  \overline v.
	\end{split}
\]
The rest of the proof is the estimate of the last line. Using convexity, we find
\begin{align*}
	&Ay^{1-r} \left( ( 1 +y^{-1}\tau^{-\frac{\gamma}{\beta}} (-R_0 - \tau \Theta  + \tau \ophim(y)))^{-(r-1)} - 1\right)	\\
	&\qquad\geq Ay^{1-r} (r-1)y^{-1}\tau^{-\frac{\gamma}{\beta}} (R_0 + \tau \Theta - \tau \ophim(y))
	\geq - A(r-1) y^{-r}\tau^{1-\frac{\gamma}{\beta}} \ophim(y) .
\end{align*}
Now, applying \Cref{p:phi} and use that $\dot S = \Theta + \dot R\geq \Theta$ yields
\begin{align*}
	\tau^{-1} &\ophim_{yy}
		+ \beta \dot S \tau^{1 - \frac{\gamma}{\beta}} \ophim_y - (r-1) y^{-r}\tau^{1-\frac{\gamma}{\beta}} \ophim(y) \\
	&\geq \tau^{-1}\left(- C\1_{[C^{-1},\infty]}
		+ C^{-1} y^{-\frac{1+r}{2}} \1_{[0,C^{-1}]}\right)
			- \Theta \tau^{1-\frac{\gamma}{\beta}} y^\frac{1-r}{2}
		\\&\qquad
		- (r-1) \tau^{1-\frac{\gamma}{\beta}} \left(Cy^{\frac{1-r}{2}}\1_{[C^{-1},\infty]}
		+ C^{-1} y^{-r} \1_{[0,C^{-1}]}\right)
	\geq -C \max\Big\{\tau^{-1},\tau^{1-\frac{\gamma}{\beta}}, \tau^{-\frac{(r-1)^2}{3-r}}\Big\}.
%
%
\end{align*}
In the last step above, we used Young's inequality with $p = (1+r)/(r-1)$ to absorb the negative $y^{(1-r)/2}$ term for small $y$ into the positive $y^{- (1+r)/2}$ term.

Combining all above estimates, we conclude that
\[
	\begin{split}
&\beta \dot R
				+ \tau^{-1}  \ophim_{yy}
				- \beta \dot S\tau^{1 - \frac{\gamma}{\beta}}\ophim_y
				+ A( y +\tau^{-\frac{\gamma}{\beta}} (-R_0 - \tau \Theta   + \tau \ophim\left(y\right)))^{-(r-1)} - Ay^{-(r-1)}\\
 &\geq \beta \dot R
				- C \max\{\tau^{-1},\tau^{1-\frac{\gamma}{\beta}}, \tau^{-\frac{(r-1)^2}{3-r}}\}.
	\end{split}
\]
The proof is then finished by choosing $R$ such that $\beta \dot R 
		= C \max\{\tau^{-1},\tau^{1-\frac{\gamma}{\beta}}, \tau^{-\frac{(r-1)^2}{3-r}}\}.$
\end{proof}

We now use \Cref{l:supersolution_small_r} to conclude the proof of the upper bound in \Cref{thm:main_delay}.(iii).

\begin{proof}[Proof of~\eqref{e:alg_delay_above}]
Let $\bar v$ be the supersolution of~\eqref{e:v_supersolution} constructed in \Cref{l:supersolution_small_r} with $R_0$ to be determined.

First we show that $\bar v \geq v$ (recall the definition of $v$ in~\eqref{e.c844}).  By the comparison principle, we have that this ordering holds as long as it holds on the parabolic boundary of the domain $\mathcal{P} = [1,\infty)\times (0,\infty).$  There are two components to this: $\cB_1 = \{1\}\times (0,\infty)$ and $\cB_2 = [1,\infty)\times \{0\}$.

We consider the case $\cB_1$ first.  By \Cref{l.heat_bounds} and the change of variables defining $v$, we find, for $y\geq 0$,
\[
	v(1,y)
		= \nu^{-1} e^{y} u(1, y + 2 - \Theta + R(1))
		\leq \frac{C}{(y + 2 - \Theta + R(1))_+ + 1} e^{y+ 1 - \frac{(y + 2 - \Theta - R(1))^2}{4}}.
\]
On the other hand, by \Cref{p:phi}, we have
\[
	\overline v(1,y)
		= \exp\left\{R_0 + \Theta - \phi(y)\right\}
		\geq \exp\left\{R_0 + \Theta - C - C y^\frac{1+r}{2}\right\}.
\]
It is clear that we may choose $R_0$ sufficiently large so that $v(1,y) \leq \overline v(1,y)$ for all $y\geq 0$.

Next, we consider the case $\cB_2$.  Since $u \leq 1$, it follows that $v(\tau, 0) \leq \nu^{-1}$.  On the other hand, $\overline v(\tau,0) = e^{R_0}$.  Hence, after increasing $R_0$ so that $e^{R_0} > \nu^{-1}$, we have $v(\tau,0) \leq \overline v(\tau,0)$.

The previous two paragraphs established that $v \leq \overline v$ on the parabolic boundary of $\mathcal{P}$, which implies that $v \leq \overline v$ on $\mathcal{P}$.  We now obtain the the bound on the front location using this inequality.

	Fix $x_0 >0$ to be chosen.  We have
\be\label{e.c845}
\begin{split}
		\lim_{t\to\infty}
			&\sup_{x \geq 2t - s(t)+x_0}
				u(t,x)
			= \nu  \lim_{t\to\infty}
			\sup_{x \geq x_0} e^{-x} w(t,x)\\
		&= \nu  \lim_{\tau \to\infty}
			\sup_{y \geq x_0\tau^{-\frac{\gamma}{\beta}}} e^{-y \tau^{\frac{\gamma}{\beta}}} w(\tau^{\frac{1}{\beta}},y \tau^{\frac{\gamma}{\beta}}) = \nu  \lim_{\tau \to\infty}
			\sup_{y \geq x_0\tau^{-\frac{\gamma}{\beta}}} e^{-y \tau^{\frac{\gamma}{\beta}}} v(\tau, y)\\	
		&\leq   \nu  \lim_{\tau \to\infty}
			\sup_{y \geq x_0\tau^{-\frac{\gamma}{\beta}}} e^{-y \tau^{\frac{\gamma}{\beta}}} \overline v(\tau, y) = \nu  \lim_{\tau \to\infty}
			\sup_{y \geq x_0\tau^{-\frac{\gamma}{\beta}}} e^{-y \tau^{\frac{\gamma}{\beta}}} \exp\left\{R_0 + \tau (\Theta -  \ophim(y))\right\}\\
		&\leq \nu  \lim_{\tau \to\infty} \exp\left\{R_0 + \tau \Theta\right\}
			\sup_{y \geq x_0\tau^{-\frac{\gamma}{\beta}}} e^{-y \tau^{\frac{\gamma}{\beta}} - \tau \ophim(y) }.
\end{split}
\ee

One can estimate
\[
	\inf_{y \geq x_0\tau^{-\frac{\gamma}{\beta}}} (y \tau^{\frac{\gamma}{\beta}} + \tau \ophim(y))
\]
in the following way. Given that $\phi$ is increasing for large $y$, the infimum is either attained at $y = x_0 \tau^{-\gamma/\beta}$ or it is attained at a finite value $y_\tau> x_0 \tau^{-\gamma/\beta}$.  We rule out the latter now.  Indeed, if attained at an interior point $y_\tau$, such a point necessarily satisfies $\ophim_y(y_\tau) = - \tau^{\frac{\gamma}{\beta}-1}$.  From~\eqref{e:phi}, which implies
\be\label{e.c846}
	\phi_y(y)
		\sim -\sqrt A y^\frac{1-r}{2},
\ee
we find $y_\tau \sim A^\frac{1}{r-1} \tau^{-\frac{\gamma}{\beta}}$. Thus, choosing $x_0 \geq 2 A^\frac{1}{r-1}$, the infimum occurs at the boundary $x_0\tau^{-\frac{\gamma}{\beta}}$ as claimed above.

Hence,~\eqref{e.c845} simplifies to:
\[
		\limsup_{t\to\infty}
			\sup_{x \geq 2t - s(t)+x_0}
				u(t,x)
			\leq \nu  \lim_{\tau \to\infty} \exp\left\{R_0 -x_0 + \tau \Theta - \tau \ophim(x_0\tau^{-\frac{\gamma}{\beta}})\right\}.
\]
Using again the asymptotics of $\phi_y$~\eqref{e.c846} to yield $\phi(y) \sim \Theta - 2 \sqrt A x^\frac{3-r}{2}/(3-r)$, we find
\[
	\limsup_{t\to\infty}
			\sup_{x \geq 2t - s(t)+x_0}
				u(t,x)
		\leq \nu \exp\left\{R_0 -x_0 + \frac{2A^\frac12}{3-r} x_0^{\frac{3-r}{2}}\right\}.
\]
The right hand side clearly tends to zero as $x_0\to\infty$.  Hence, recalling that $\Theta = \Theta_r A^\gamma$ by scaling, the proof is finished.
%
%
%
%

\end{proof}

\subsection{A lower bound on the front location}

%
%
%
%
%
%


We now complete the proof of \Cref{thm:main_delay}.(iii) by proving the lower bound~\eqref{e:alg_delay_below}.  We do so by constructing an appropriate subsolution of~\eqref{e:main} using $\ophic$, which was constructed in \Cref{s:phi}.

We work in the shifted frame and try to build a sub-solution. Write
\begin{equation*}
u(t,x+2t) = \nu e^{-x} w(t,x).
\end{equation*}
The function $w$ then satisfies
\begin{equation}\label{eq:w}
w_t = w_{xx} - A\left(x + \log(1/w)\right)^{1-r}w.
\end{equation}

\begin{lemma}\label{l.subsolution_sr}
There exists $p>0$, $\epsilon_0>0$ and a decreasing $\cC_{\rm loc}^1$ function $h$, such that $h(2^{-\beta}) = 0$ and, if $\e \in(0,\e_0)$ the function 
\begin{equation*}
\underline w(t,x) =  \epsilon \left(\frac{x}{t^\gamma}\right)^p  \exp\left\{ h\big(t^\beta \big) - t^\beta \ophic\left(x t^{-\gamma}\right) \right\}
\end{equation*}
is a sub-solution of \eqref{eq:w} on the domain $\left\lbrace (t,x), t \geq 1/2, x \geq 0\right\rbrace$.  Further $h(t)/ t \to 0$ as $t\to\infty$.
%
\end{lemma}

\begin{proof}
We use the same change of variables as in \Cref{l:supersolution_small_r}; that is, letting
\[
	\underline z(\tau,y)
		= \underline w(\tau^{1/\beta}, y \tau^{\gamma/\beta})
		= \epsilon y^p \exp\left\{h(\tau) - \tau \Phi(y)\right\}.
\]
it suffices to show that $\cL \underline z \leq 0$ where
\[
	\cL z
		= \beta z_\tau - \frac{\gamma}{\tau} y z_y - \frac{1}{\tau^2} z_{yy}
			+ A \left(y - \tau^{-\frac{\gamma}{\beta}} \log z \right)^{1-r} z.
\]

We now compute $\cL$ and use that $- \beta \Phi + \gamma y \Phi_y - |\Phi_y|^2 + A y^{1-r} =0$ to find
\be\label{e.c847}
	\begin{split}
		\underline{z}^{-1} &\mathcal{L}(\underline{z})
			= \beta h'(\tau)
				- \beta \ophic
				- \gamma \tau^{-1} p
				+ \gamma  y \ophic_y
				- \tau^{-2} \frac{p(p-1)}{y^2} + 2 \tau^{-1}  
				\frac{p}{y} \ophic_y\\
				&\qquad\qquad +  \tau^{-1} \ophic_{yy}
					-\vert \ophic_y \vert^2
					+ A\left(y - (1+\beta \tau)^{-\frac{\gamma}{\beta}} [  \log(\underline z)]\right)^{1-r}\\
			&= \beta h'(\tau) - \gamma p   \tau^{-1} 
- \tau^{-2} \frac{p(p-1)}{y^2} + \tau^{-1} \left[2 \frac{p}{y}  
 \ophic_y +\ophic_{yy} \right] + A\left(y- \tau^{-\frac{\gamma}{\beta}}  \log(\underline z)\right)^{1-r} - Ay^{1-r}\\
			 &= \beta h'(\tau)
			 	+ \left[2 \frac{p}{y}  
 \ophic_y +\ophic_{yy} - \gamma p - \frac{p(p-1)}{\tau y^2}\right]  \tau^{-1}  + A\left[\left(y- \tau^{-\frac{\gamma}{\beta}}  \log(\underline z)\right)^{1-r} - y^{1-r}\right].
	\end{split}
\ee
The two bracketed terms in the last line of~\eqref{e.c847} require bounds.

We begin with the first bracketed term.  Using \Cref{p:Phi}, we find
\[
	2 \frac{p}{y} \ophic_y +\ophic_{yy}
 		\leq  - \frac{p}{C} y^{-\frac{r+1}{2}} + Cp + C\left(1 + y^{-\frac{1+r}{2}}\right).
\]
Choosing $p$ sufficiently large, we find $2 \frac{p}{y} \ophic_y +\ophic_{yy} \leq C$.  Choosing $h$ such that $h(2^{-\beta}) = 0$ and $\beta h' = -C/\tau$ yields
\be\label{e.c848}
	\beta h'(\tau) 
		+ \left[2 \frac{p}{y} \ophic_y +\ophic_{yy}\right]\tau^{-1}
		\leq 0.
\ee

Next, we consider the second bracketed term.  Notice that, for $\tau \geq 2^{-\beta}$,
\[
	-\log(\underline z)
		= \log(1/\epsilon)
			- p \log(y)
			- h(\tau)
			+ \tau \Phi(y)
		\geq 0.
\]
The last inequality follows, after possibly decreasing $\e$, from \Cref{p:phi} and \Cref{p:Phi}, which imply that $\Phi(y) \geq \min\{C^{-1}, y^2/4\}$.  Hence,
\be\label{e.c849}
	\left( y - \tau^{-\frac{\gamma}{\beta}} \log(\underline z)\right)^{1-r}
		- y^{1-r} \leq 0.
\ee

Using the bound~\eqref{e.c848} and~\eqref{e.c849} in~\eqref{e.c847}, we conclude that
\[
	\underline z^{-1} \cL(\underline z) \leq 0.
\]
which concludes the proof.
\end{proof}

We now conclude the proof of the lower bound in \Cref{thm:main_delay}.(iii) using \Cref{l.subsolution_sr}.

\begin{proof}[Proof of \eqref{e:alg_delay_below}]
Let $\underline w$ be the function defined in \Cref{l.subsolution_sr} with $\epsilon$ to be determined.  We shift $\underline w$, letting
\[
	\underline w_{\rm shift}(t,x)
		= \underline w(t-1/2, x).
\]
This shift allows us to more easily ``fit'' $\underline w_{\rm shift}$ under $w$ at time $t=1$.

We aim to use the comparison principle to show that $\underline w_{\rm shift} \leq w$ as long as $\e$ is sufficiently small.  Since~\eqref{eq:w} is autonomous, we have that $\underline w_{\rm shift}$ is a subsolution of~\eqref{eq:w} on $\cP = \{(t,x) : t > 1, x > 0\}$.  Hence, we need only check the ordering of $\underline w_{\rm shift}$ and $w$ on the parabolic boundary of $\cP$.

Since $u$, and thus $w$, is positive for any positive time, the ordering at $x=0$ is straightforward for any $t \geq 1$.  Hence, we need only check the portion of the parabolic boundary when $t=1$.  Using \Cref{l.heat_bounds}, we find
\[
	w(1,x)
		= \nu^{-1} e^x u(1, x + 2)
		\geq \frac{1}{\nu C (x + 3)} e^{x - \frac{(x + 2)^2}{4} - C(x+2)}
		\geq \frac{1}{\nu C (x + 3)} e^{ - \frac{x^2}{4} - Cx}.
\]
On the other hand, we have, using \Cref{p:Phi},
\[
	\underline w_{\rm shift}(1,x)
		= \underline w(1/2,x)
		= 2^{\gamma p} \e x^p e^{- 2^{-\beta} \Phi(x 2^\gamma)}
		\geq 2^{\gamma p} \e x^p e^{- 2^{-\beta} \frac{(x 2^\gamma)^2}{4}}
		= 2^{\gamma p} \e x^p e^{- \frac{x^2}{2}}.
\]
It follows that, up to decreasing $\e$, we have $\underline w_{\rm shift} \leq w$ at $t=1$.  We conclude that $\underline w_{\rm shift} \leq w$ on $\cP$ from the comparison principle.

%

This yields a decaying bound beyond the front, so in order to obtain a bound near the front, we ``trace back'' with a traveling wave as we did in the proof of~\eqref{e:weak_log_delay_below} (see also \cite[Section 3]{HNRR_Short}).  As such, we provide only a brief outline.

Fix any $\bar \e \in (0,1)$ and any $\overline \Theta > \Theta$.  Define $U_{\bar\e}$ analogously as in~\eqref{e.c7242} but with $-2$ replaced by $1-r$.  Let $L>0$ be a constant to be chosen and
\[
	\underline u(t,x) = U(x - 2t + \overline \Theta t^\beta + L).
\]
As in the proof of~\eqref{e:weak_log_delay_below}, $\underline u$ is a subsolution of~\eqref{e:main} on $[t_0, \infty)\times \R$ for some $t_0\geq 1$.

Let $\ell(t)$ be a $\cC^1_{\rm loc}$ function to be chosen such that $\ell(t) \to 0$ and $\ell(t)t^\gamma$ is increasing to infinity as $t\to\infty$.  We show that $\underline u \leq u$ on $\cP = \{(t,x) : t > t_0, x < 2t + \ell(t)t^\gamma\}$ by the comparison principle.  To achieve this, we need only check that $\underline u \leq u$ on the parabolic boundary of $\cP$.  We check this now.

First, we examine the portion of the parabolic boundary of $\cP$ where $t = t_0$. By construction $\liminf_{x\to-\infty} U_{\bar\e}(x) < 1$.  Hence, up to increasing $t_0$, we have that
\[
	\liminf_{x\to-\infty} u(t_0,x)  > \liminf_{x\to-\infty} U_{\bar \e}(x).
\]
After possibly increasing $L$, which ``shifts'' $U_{\bar \e}$ to the left, we find that $\underline u(t_0,\cdot) \leq u(t_0,\cdot)$ on $(-\infty, 2 t_0 + \ell(t_0)t_0^\gamma)$, as desired.

Next, we check the portion of the parabolic boundary where $x = 2t_0 + \ell(t_0)$.  Using \Cref{p:decayTW} and possibly increasing $L$, we find
\be\label{e.c851}
\begin{split}
	\underline u(t, 2t + \ell(t)t^\gamma)
		&= U_{\bar \e}(\ell(t)t^\gamma + \overline \Theta t^\beta + L)\\
		&\leq \exp\left\{ - \ell(t)t^\gamma - \overline \Theta t^\beta - L + \frac{4}{3-r} \sqrt A \left(\ell(t)t^\gamma + \overline \Theta t^\beta + L\right)^\frac{3-r}{2}\right\}.
\end{split}
\ee
On the other hand, using~\eqref{e.c7242} and that $h$ is decreasing, we have, for all $t\geq t_0$,
\be\label{e.c852}
\begin{split}
	u(t, &2t + \ell(t)t^\gamma)
		= \nu e^{-\ell(t)t^\gamma} w(t, \ell(t)t^\gamma)
		\geq \nu e^{-\ell(t)t^\gamma} \underline w_{\rm shift}(t, \ell(t)t^\gamma)\\
		&= \nu \left(\frac{\ell(t)t^\gamma}{(t-1/2)^\gamma}\right)^p \exp\left\{- \ell(t)t^\gamma + h\left((t-1/2)^\beta\right) - (t-1/2)^\beta \Phi\left(\frac{\ell(t) t^\gamma}{(t-1/2)^\gamma}\right)\right\}\\
		&\geq \frac{1}{C}
			\left(\ell(t)\right)^p
			\exp\left\{
				- \ell(t)t^\gamma
				+ h(t^\beta)
				- \Theta t^\beta
				+ \frac{1}{3-r} \sqrt A \ell(t)^\frac{3-r}{2} t^\beta
				\right\}
\end{split}
\ee
The last inequality holds as long as $\ell(t_0)$, and thus $\ell(t)$, is sufficiently small that the asymptotics of~\eqref{e.c7242} hold.  
Hence, after rearranging~\eqref{e.c851} and~\eqref{e.c852}, we have $\underline u(t,2t+\ell(t)) \leq u(t,2t+\ell(t))$ as long as
\[\begin{split}
	- \overline \Theta t^\beta - L &+ \frac{4}{3-r} \sqrt A \left(\ell(t)t^\gamma + \overline \Theta t^\beta + L \right)^\frac{3-r}{2}\\
		&\leq - \log(C) - p \log(1/\ell(t)) + h(t^\beta) - t^\beta \Theta + \frac{1}{3-r} \sqrt A \ell(t)^\frac{3-r}{2},
\end{split}\]
which, rearranged, is equivalent to
\[
	\frac{p \log(1/\ell(t)) 
		- h(t^\beta)}{t^\beta}
		 + \frac{\sqrt A}{3-r} \left(4\left(\ell(t) + \frac{\overline \Theta t^\beta + L}{t^\gamma} \right)^\frac{3-r}{2} - \ell(t)^\frac{3-r}{2}\right)
		\leq  \left(\overline \Theta - \Theta\right) + \frac{L- \log(C)}{t^\beta}.
\]
Recall that $h(t^\beta)/t^\beta \to 0$ as $t\to\infty$ by \Cref{l.subsolution_sr}.  Hence, such a condition holds if $L$ is chosen sufficiently large and $\ell(t)$ is defined to be, say, $t^{\beta - \gamma}$.  We conclude that $\underline u \leq u$ on the parabolic boundary of $\cP$, and, thus, we find that $\underline u \leq u$ on $\cP$ by the comparison principle.

%
%
%
We finish the proof by noting that:
\begin{align*}
	\liminf_{t \to \infty} &\inf_{x \leq 2t - \left(\Theta + 2(\overline \Theta - \Theta)\right) t^\beta} u(t,x)
		= \liminf_{t \to \infty} \inf_{x \leq -(\overline \Theta - \Theta)t^\beta} u(t,x + 2t - \overline \Theta t^\beta)\\
		&\geq \liminf_{t \to \infty} \inf_{x \leq -(\overline \Theta - \Theta)t^\beta} \underline u(t,x + 2t - \overline \Theta t^\beta)
		= \liminf_{t \to \infty} \inf_{x \leq -(\overline \Theta - \Theta)t^\beta} U_{\bar \e}(x + L)\\
		&= \liminf_{t\to\infty} U_{\bar \e}\left(- (\overline \Theta - \Theta) t^\beta + L\right)
		= \lim_{x\to-\infty} U_{\bar \e}(x).
\end{align*}
The last line follows because $U_{\bar \e}$ is decreasing (see \Cref{p:decayTW}).  The result follows from the arbitrariness of $\bar \e$, the fact that $\lim_{x\to-\infty} U_{\bar \e}(x) \to 1$ as $\bar \e \to 0$, the arbitrariness of $\overline \Theta$, and the fact that $\Theta = \Theta_r A^\gamma$, observed above.
\end{proof}

\section{A technical lemma: the weak small time bounds on $u$}\label{s:technical}


\begin{proof}[Proof of \Cref{l.heat_bounds}]
For both the upper and lower bounds, we require $h$ satisfying $h_t = h_{xx}$ in $(0,\infty)\times \R$ with the initial data $u(0,\cdot) = u_0$.  We begin with the lower bound.

{\bf The lower bound.} 
Let $\underline u = h$.  Clearly $\underline u$ is a subsolution of~\eqref{e:hb} since $f \geq 0$.  Hence, by the comparison principle, $\underline u \leq u$.  Using the heat kernel, we thus find, for any $x$,
\[\begin{split}
	u(t,x)
		\geq \underline u(t,x)
		= \int \frac{e^{-\frac{(x-y)^2}{4t}}}{\sqrt{4\pi t}} u_0(y) dy.
\end{split}\]
From~\eqref{e:u_0}, we have $\epsilon>0$ such that $u_0 \geq \epsilon \1_{(-\infty,-1/\epsilon)}$.  Hence,
\[\begin{split}
	u(t,x)
		&\geq \e \int_{-\infty}^{-1/\e} \frac{e^{-\frac{(x-y)^2}{4t}}}{\sqrt{4\pi t}} dy.
\end{split}\]

If $x\leq \sqrt t$, we find
\[
	u(t,x)
		\geq \e \int_{-\infty}^{-\sqrt t - \frac{1}{\e}}\frac{e^{-\frac{y^2}{4t}}}{\sqrt{4\pi t}} dy
		= \e \int_{-\infty}^{-1 -\frac{1}{\e\sqrt t}} \frac{e^{-\frac{y^2}{4}}}{\sqrt{4 \pi}} dy
		\geq \e \int_{-\infty}^{-1-\frac{1}{\e}} \frac{e^{-\frac{y^2}{4}}}{\sqrt{4 \pi}} dy,
\]
where we used that $t\geq 1$.  This yields the claim in this case by choosing $C$ sufficiently large.

If $x > \sqrt t$, we instead find
\[\begin{split}
	u(t,x)
		&\geq \e \int^{\infty}_{x+1/\e} \frac{e^{-\frac{y^2}{4t}}}{\sqrt{4\pi t}} dy
		= \e \int_{\frac{x+1/\e}{\sqrt t}}^{\infty} \frac{e^{-\frac{y^2}{4 }}}{\sqrt{4\pi }} dy
		\geq \e \int_{\frac{x+1/\e}{\sqrt t}}^{\frac{x+1/\e}{\sqrt t} + \frac{\sqrt t}{x+1/\e}} \frac{e^{-\frac{y^2}{4 }}}{\sqrt{4\pi }} dy\\
		&\geq \e \int_{\frac{x+1/\e}{\sqrt t}}^{\frac{x+1/\e}{\sqrt t} + \frac{\sqrt t}{x+1/\e}} \frac{e^{- \frac{(x+1/\e)^2}{4t} - \frac{1}{2} - \frac{\sqrt t}{4(x+1/\e)^2} }}{\sqrt{4\pi}} dy
		\geq \frac{\e e^{-2}}{\sqrt{4\pi}} \frac{\sqrt t}{x + \frac{1}{\e}} e^{-\frac{(x+1/\e)^2}{4t}}.
\end{split}\]
The claim is then finished by choosing $C$ sufficiently large depending on $\e$.  This concludes the proof of the lower bound.

{\bf The upper bound.}
The arguments for the upper bound are similar, except that we have the upper bound on the initial data:
\[
	u_0
		\leq \1_{(-\infty,0)}.
\]
When $x \leq \sqrt t$, we have $u(t,x) \leq 1$, so we need only consider the case $x\geq \sqrt t$.

We construct a supersolution bounding $u$ from above.  Indeed, since $f(s) \leq s$ for all $s\in[0,1]$, we note that $\overline u(t,x) = e^t h(t,x)$ is a supersolution of~\eqref{e:hb} as it satisfies $\overline u_t = \overline u_{xx} + \overline u$.  Hence, $u \leq \overline u$.  It follows that we need only bound $\overline u$ from above.  In fact, it is clear that we only require a bound on $h$ since the integrating factor $e^t$ is explicit.

Arguing exactly as above, we find
\[
	h(t,x)
		\leq \frac{1}{\e} \int_{-\infty}^{- \frac{x}{ \sqrt t}} \frac{e^{-\frac{y^2}{4}}}{\sqrt{4\pi}} dy
		\leq \frac{1}{\e} \int_{\frac{x}{ \sqrt t}}^{\infty} \frac{e^{-\frac{y}{4} \frac{x}{\sqrt t}}}{\sqrt{4\pi}} dy
		\leq \frac{1}{\e \sqrt{4\pi}} \frac{\sqrt t}{x} e^{- \frac{x^2}{4t}},
\]
which concludes the proof.
\end{proof}

 \bibliographystyle{abbrv}
  \bibliography{refs}
\end{document}